\documentclass[leqno]{amsart}
\usepackage{amsfonts,amsmath,enumerate,latexsym,verbatim,amscd,mathrsfs,color,array,dsfont,appendix}%,newclude}

\usepackage[left,modulo,pagewise, mathlines]{lineno}
\modulolinenumbers[1]

\usepackage[latin1]{inputenc}
\newtheorem{theorem}{Theorem}[section]
\newtheorem{remark}[theorem]{Remark}
\newtheorem{proposition}[theorem]{Proposition}
\newtheorem{hypothesis}[theorem]{Hypothesis}
\newtheorem{lemma}[theorem]{Lemma}
\newtheorem{definition}[theorem]{Definition}

\usepackage[colorlinks=true]{hyperref}
\hypersetup{urlcolor=blue,linkcolor=black,citecolor=black,colorlinks=true}

\DeclareMathAlphabet{\mathpzc}{OT1}{pzc}{m}{it}

\title[Coalescence - Fragmentation processes]{Stochastic Coalescence Multi-Fragmentation processes}
\author[Eduardo CEPEDA (LAMA)]{Eduardo CEPEDA} %\\ 

\address{Laboratoire d'Analyse et de Math\'ematiques Appliqu\'ees, UMR 8050. Universit\'e Paris-Est. 61, avenue du G\'en\'eral de Gaulle, 94010 Cr\'eteil C\'edex}
\email{eduardo.cepeda@m4x.org}

\begin{document}

\begin{abstract}

We study infinite systems of particles which undergo coalescence and fragmentation, in a manner determined solely by their masses. A pair of particles having masses $x$ and $y$ coalesces at a given rate $K(x,y)$.  A particle of mass $x$ fragments into a collection of particles of masses $\theta_1 x, \theta_2 x, \ldots$ at rate $F(x) \beta(d\theta)$.  We assume that the kernels $K$ and $F$ satisfy  H\"older regularity conditions with indices $\lambda \in (0,1]$ and $\alpha \in [0, \infty)$ respectively.  We show existence of such infinite particle systems as strong Markov processes taking values in $\ell_{\lambda}$, the set of ordered sequences $(m_i)_{i \ge 1}$ such that $\sum_{i \ge 1} m_i^{\lambda} < \infty$. We show that these processes possess the Feller property.  This work relies on the use of a Wasserstein-type distance, which has proved to be particularly well-adapted to coalescence phenomena.\medskip

\medskip

\end{abstract}
\maketitle

%\tableofcontents
\begin{center}
To appear in \it{Stochastic Processes And Their Applications}
\end{center}
\section{Introduction}
\setcounter{equation}{0}

A coalescence-fragmentation process is a stochastic process which models the evolution in time of a system of particles undergoing coalescence and fragmentation. The size of a particle increases and decreases due to successive mergers and dislocations.  We assume that each particle is fully identified by its mass $x \in (0,\infty)$.  We consider the mean-field setting, so that the positions of particles in space, their shapes, and other geometric properties are not considered.  Examples of applications of these models arise in the study of polymers, aerosols and astronomy; see the survey papers  \cite{Drake,Aldous} for more details.\medskip

In this paper, we will concern ourselves with the phonomena of coalescence and fragmentation at a macroscopic scale.  Consider a (possibly infinite) system of particles.  The framework we consider is as follows.  The coalescence of particles of masses $x$ and $y$ results in the formation of a new particle of mass $x+y$.  We assume that this coalescence occurs at rate $K(x,y)$, where $K$ is some symmetric coagulation kernel.  Particles may also fragment: we assume that a particle of mass $x$ splits into a collection of particles of smaller masses $\theta_1 x, \theta_2 x, \ldots$ at rate $F(x) \beta(d\theta)$.  Here, $F: (0,\infty) \to [0,\infty)$ and $\beta$ is a positive measure on the set $\Theta := \{\theta = (\theta_i)_{i \ge 1}: 1 > \theta_1 \ge \theta_2 \ge \ldots \ge 0\}$.  In particular, this means that the distribution of the ratios of the masses of the child particles to the mass of the parent particle is a function of these ratios only, and not of the mass of the parent particle.  In this setting, our coalescence-fragmentation processes will be defined through their infinitesimal generators.  Note that at a fixed time the state may be composed of an infinite number of particles which have finite total mass.\medskip

In previous works by Evans and Pitman \cite{Evans_Pitman}, Fournier \cite{Sto-Coal} and Fournier and L\"ocherbach \cite{Sto-Coal2}, pure stochastic coalescents with an infinite number of particles have been constructed for a large class of kernels $K$. See also the survey paper by Aldous \cite{Aldous}. On the other hand, the fragmentation model we study was first introduced by Bertoin \cite{Bertoin2} where the author takes into account an infinite measure $\beta$ and a mechanism of dislocation with a possibly infinite number of fragments. The properties of the only fragmentation model are studied in Bertoin \cite{Bertoin2,Bertoin} and in Hass \cite{Haas,Haas2}. We refer also to the book \cite{Bertoin_Book} where a extensive study of coalescence and fragmentation is carried out.\medskip 

The present paper combines the two phenomena.  We are mainly concerned with a general existence and uniqueness result, and seek to impose as few conditions on $K$, $F$ and $\beta$ as possible.  Roughly speaking, our assumptions are that the coalescence and fragmentation kernels each satisfy a H\"older regularity condition which makes them bounded near the origin ($(0,0)$ and $0$ respectively) but not near $\infty$.  The measure $\beta$ is allowed to be infinite.\medskip

We follow ideas developed in \cite{Sto-Coal,Sto-Coal2} in the context of pure coalescence ($F \equiv 0$).  We work on the set $\ell_{\lambda}$ of ordered sequences of non-negative real numbers $(m_i)_{i \ge 1}$ which are such that $\sum_{i \ge 1} m_i^{\lambda} < \infty$. We endow this space with a Wasserstein-type distance $\delta_{\lambda}$: for $m, \tilde{m} \in \ell_{\lambda}$, let
\begin{equation*}
\delta_{\lambda}(m, \tilde{m}) = \inf_{\pi,\sigma \in \mathrm{Perm}(\mathbb{N})} \sum_{i \ge 1} \left| m_{\pi(i)}^{\lambda} - \tilde{m}_{\sigma(i)}^{\lambda} \right|,
\end{equation*}
where $\mathrm{Perm}(\mathbb{N})$ denotes the set of finite permutations of $\mathbb{N}$.\medskip

Extending results in  \cite{Sto-Coal,Sto-Coal2}, we construct a stochastic particle system undergoing coalescence and fragmentation. In Theorem  \ref{Result_Theo}, we show existence and uniqueness of a stochastic coalescence-fragmentation process as a Markov process in $\mathbb{D}([0,\infty), \ell_{\lambda})$ which enjoys the Feller property. We use a convergence method, starting from a finite process, for which the initial number of particles in the system is bounded, and fragmentation occurs at a bounded rate and produces a bounded number of fragments.  Existence and uniqueness are obtained for these finite processes in a straightforward manner.  We pass to the limit using a Poisson-driven stochastic differential equation (SDE) associated to the model, and coupling techniques.\medskip

In the finite case (see Proposition \ref{Def_finiteproc}), we require only that the coalescence and fragmentation kernels be locally bounded (in the sense that for all $a > 0$, $\sup_{(0,a]^2} K(x,y) <\infty$ and $\sup_{(0,a]} F(x) < \infty$) and that there be only a finite number of particles.  In order to extend the results to a system composed of an infinite number of particles and with fragmentation into an infinite number of fragments, it is necessary to impose the additional continuity conditions on the kernels.\medskip

The first works known to us on coalescence-fragmentation processes focussed on binary fragmentation, where a particle may only split into \emph{two} child particles: denoting $c_t(x)$ the concentration of particles of mass $x\in(0,\infty)$ at time $t$, in this case the dynamics is given by

\begin{eqnarray}\label{binary_Detfrag}
\partial_t c_t(x)  & = &  \frac{1}{2}\int_0^{x} K(y,x-y) c_t(y)c_t(x-y)\, dy - c_t(x)
\int_0^{\infty}K(x,y) c_t(y)\,dy \\ 
&& +   \int_x^{\infty} F_b(x,y-x)c_t(y) dy - \frac{1}{2} c_t(x) \int_0^x
F_b(y,x-y) dy, \notag
\end{eqnarray}

The binary fragmentation kernel $F_b(\cdot,\cdot)$ is a symmetric function and $F_b(x,y)$ gives the rate at which a particle of mass $x+y$ fragments into particles of masses $x$ and $y$.  In this case, the total fragmentation rate of a particle of mass $x$ is given by $\frac{1}{2} \int_0^x F_b(y,x-y) dy$. \medskip

The deterministic setting of our model has been studied in Cepeda \cite{CepedaDet} where existence and uniqueness of the corresponding equation is proved, using the same notation as for equation (\ref{binary_Detfrag}), the equation reads as follows:

\begin{eqnarray}\label{mult_Detfrag}
\partial_t c_t(x) & = &   \frac{1}{2}\int_0^{x} K(y,x-y) c_t(y)c_t(x-y)\, dy - c_t(x)
\int_0^{\infty}K(x,y) c_t(y)\,dy\\
&   & + \int_{\Theta} \Big [\sum_{i=1}^{\infty}
\frac{1}{\theta_i}F\left(\tfrac{x}{\theta_i}\right)c_t\left(\tfrac{x}{\theta_i}\right) -
F(x)c_t(x) \Big ]\beta(d\theta).  \notag
\end{eqnarray}

Note that we can obtain the continuous coagulation binary-fragmentation equation (\ref{binary_Detfrag}), for example, by considering $\beta$ with support in $\{\theta: \theta_1 + \theta_2 = 1\}$ and $\beta(d\theta) = h(\theta_1) d\theta_1 \delta_{\{\theta_2 = 1-\theta_1\}}$ and setting $F_b(x,y) = \frac{2}{x+y} F(x+y) h\left(\frac{x}{x+y}\right)$, where $h(\cdot)$ is a continuous function on $[0,1]$ which is symmetric about $1/2$.\medskip

In the binary framework, and under the additional assumptions that the kernel $K$ and the total fragmentation rate are bounded, some results on existence, uniqueness and convergence to the solution of the deterministic equation  (\ref{binary_Detfrag}) may be found in Guia\c{s} \cite{Guias}. Jeon \cite{Jeon} considered the discrete coagulation-fragmentation equation. He showed that the weak limit points of the stochastic particle system exist and provide a solution.  He assumed that $K(x,y) = o(x) o(y)$ and that the total rate of fragmentation of a particle of mass $x$ is $o(x)$.\medskip

In Fournier and Giet  \cite{FournierGiet}, the authors study the behaviour of small particles in the coagulation-fragmentation equation (\ref{binary_Detfrag}) using a probabilistic approach. They assume a linear bound on the coagulation kernel, but allow for the total fragmentation rate to be infinite.  Eibeck and Wagner \cite{Eibeck_Wagner1} proved tightness of the corresponding stochastic particle systems and characterize the weak limit points as solutions. A continuous coagulation kernel satisfying $K(x,y) = o(x)o(y)$ for $x,y \to \infty$ is required, as is a weakly continuous fragmentation measure for which the total fragmentation rate of a particle of size $x$ is $o(x)$ as $x \to \infty$.  We refer also to Eibeck and Wagner \cite{Eibeck_Wagner2}, where a general model is studied which is used to approach general nonlinear kinetic equations.\medskip

Kolokoltsov \cite{Kolokoltsov2} shows a hydrodynamic limit result for a mass exchange Markov process in the discrete case.  In Kolokoltsov \cite{Kolokoltsov1}, existence and uniqueness are proved under different assumptions to ours; there, the author assumes a multiplicative bound on the coagulation rates and a linear growth for the fragmentation rates.  For that model, the author also proves convergence to the deterministic equation.  An extensive study of the methods used by the author is given in the books \cite{Kolokoltsov_Book2,Kolokoltsov_Book}.  Finally, we also refer to Berestycki \cite{Berest}, who proves a similar result to ours for a class of exchangeable coalescence-fragmentation processes.\medskip

We believe that it is possible to obtain a \textit{hydrodynamic limit} result concerning our model: making tend simultaneously the number of particles to infinite and their sizes to $0$ may allow to prove convergence of the stochastic coalescence-fragmentation process to the solution to equation (\ref{mult_Detfrag}). Considering $d_{\lambda}$ an equivalent distance to $\delta_\lambda$ on measures; see \cite{Dist-Coal}, and $(\mu_t)_{t\geq 0}$ the solution to the deterministic equation which is a measure, we can proceed in the following way. We fix $n \in \mathbb N$, we begin by constructing a system consisting on a finite number of particles, the initial number of particles $N_0$ is set in such a way that $d_{\lambda}(\mu_0,\mu_0^n) \leq C / \sqrt{n}$; see \cite[Proposition 3.2.]{Cepeda_Fournier} where a way to construct such systems is already provided. Thus, $\mu^n_0$ is set as a discretisation of the initial condition $\mu_0$ consisting in $N_0$ atoms of weight $1/n$, this is $\mu_0^n = \frac{1}{n} \sum_{i=1}^{N_0} \delta_{m_i}$, here $\delta_{m_i}$ holds for the Dirac measure on $m_i$.\medskip

\noindent Next, we make the system $\mu_t^n$ to evolve following the dynamics of a coalescence-fragmentation process where the number of particles at each time $t>0$ is determined by the successive mergers and fragmentations, so that  $\mu^n_t = \frac{1}{n} \sum_{i=1}^{N_t} \delta_{m_i}$. This method requires some finite moments to $\mu_0$, but we believe that is possible to control $d_{\lambda}(\mu_t,\mu_t^n)$ by roughly $C_t/\sqrt{n}$ where $C_t >0$, allowing to show convergence and furthermore deduce a rate. Norris \cite{Norris,Norris2} gives a first result on convergence for the pure coalescence case $F\equiv 0$; see Cepeda and Fournier \cite{Cepeda_Fournier} for an explicit rate of convergence where the method discribed in this paragraph is applied also to pure coalescence. \medskip

The rest of this paper is organized as follows.  We introduce the notation and formal definitions in Section \ref{Notations}.  The main results may be found in Section \ref{Result}.  Stochastic coalescence-fragmentation processes are studied in Section \ref{SDEIntro}, and in Appendix \ref{appendix} we give some useful technical details.

\section[Notation and Definitions]{Notation and Definitions}\label{Notations}
\setcounter{equation}{0}

Let $\mathcal S^{\downarrow}$ be the set of non-increasing sequences $m=(m_n)_{n\geq 1}$ with values in $[0,+\infty)$. A state $m$ in $\mathcal S^{\downarrow}$  represents the sequence of the ordered masses of the particles in a particle system. Next, for $\lambda \in (0,1]$, consider
\begin{eqnarray}\label{Intro:ellset}
\ell_{\lambda} & = & \left\lbrace m = (m_k)_{k\geq 1} \in \mathcal S^{\downarrow}, \|m\|_{\lambda}:= \sum_{k=1}^{\infty} m^{\lambda}_k < \infty \right\rbrace.
\end{eqnarray}
%note that the case $\lambda = 0$ is excluded, since it corresponds to finite particles systems. \medskip

Consider also the sets of finite particle systems, completed for convenience with infinite $0$-s.
\begin{equation}
\ell_{0+}  = \left\lbrace m = (m_k)_{k\geq 1} \in \mathcal S^{\downarrow},  \inf\{k\geq 1, m_k = 0\} < \infty\right\rbrace. \nonumber
\end{equation}
\begin{remark}\label{inclusion}
Note that for all $0<\lambda_1<\lambda_2$, $\ell_{0+}\subset\ell_{\lambda_1} \subset \ell_{\lambda_2}$. Note also that, since $\|m\|_1 \leq \|m\|_{\lambda}^{\frac{1}{\lambda}}$ the total mass of $m \in \ell_{\lambda}$ is always finite.
\end{remark}

\begin{hypothesis}\label{HypSto}
We present now the general hypotheses.\medskip

\noindent \emph{Coagulation and Fragmentation Kernels.-} We consider coagulation kernel $K$, symmetric $K(x,y) = K(y,x)$ for $(x,y) \in [0,\infty)^2$ and bounded on every compact subset in $(0,\infty)^2$. There exists $\lambda \in (0,1]$ such that for all $a>0$ there exists a constant $\kappa_a > 0$ such that for all $x$, $y$, $\tilde x$, $\tilde y \in (0,a]$,
\begin{eqnarray}
|K(x,y) - K(\tilde x,\tilde y)| &\leq & \kappa_a \left[|x^{\lambda} - \tilde x^{\lambda}| + |y^{\lambda} - \tilde y^{\lambda}|\right],\label{Hyp_CoagK}
\end{eqnarray}
We consider also a fragmentation kernel $F: (0,\infty) \mapsto [0,\infty)$, bounded on every compact subset in $(0, \infty)$. There exists $\alpha \in [0, \infty)$ such that for all $a>0$ there exists a constant $\mu_a > 0$ such that for all $x,\,\tilde x\in(0,a]$,
\begin{eqnarray}
|F(x) - F(\tilde x)| \leq\mu_a\, |x^{\alpha} - \tilde x^{\alpha}|.\label{Hyp_FragK}
\end{eqnarray}

\noindent We define the set of ratios by
\begin{equation}
\Theta = \left\lbrace \theta = (\theta_k)_{k\geq 1} : 1 > \theta_1 \geq \theta_2 \geq \ldots \geq 0\, \right\rbrace. \nonumber
\end{equation} 

\noindent \emph{The $\beta$ measure.-} %]\label{Hyp_beta}
We consider on $\Theta$ a measure $\beta(\cdot)$ and assume that it satisfies 
\begin{eqnarray}
\beta\left( \sum_{k\geq 1} \theta_k > 1 \right) & = & 0, \label{Hyp1_Beta} \\
C_{\beta}^{\lambda} :=  \int_{\Theta} \left[\sum_{k\geq 2} \theta_k^{\lambda} + (1-\theta_1)^{\lambda}\right]\beta(d\theta) & <  & \infty, \hspace{5mm} \textrm{for some } \lambda\in (0,1]. \label{Hyp2_Beta}
\end{eqnarray}
\end{hypothesis}
For example, the coagulation kernels listed below, taken from the mathematical and physical literature, satisfy Hypothesis \ref{HypSto}.

\begin{align} 
K(x,y) & \equiv 1 & \textrm{(\ref{Hyp_CoagK}) holds with } \kappa_a = 0, \nonumber \\ 
K(x,y) & =   (x^\alpha + y^\alpha)^\beta & \textrm{with}\,\,\,
\alpha\in(0,\infty),\,\,\beta\in(0,\infty) \,\,\,\textrm{and}\,\,\,\lambda = \alpha\beta \in(0,1],
\nonumber \\ 
K(x,y) & =   x^\alpha y^\beta + x^\beta y^\alpha  & \textrm{with}\,\,\, 0 \leq \alpha \leq \beta \leq
1 \,\,\,\textrm{and}\,\,\,\lambda = \alpha+\beta \in(0,1], \nonumber \\ 
K(x,y) & =  (xy)^{\alpha/2} (x + y)^{-\beta}  & \textrm{with}\,\,\, \alpha \in (0,1],\,\,\,
\beta\in[0,\infty) \,\,\,\textrm{and}\,\,\,\lambda = \alpha-\beta \in(0,1], \nonumber \\ 
K(x,y) & =   (x^\alpha+y^\alpha)^{\beta} |x^{\gamma} - y^{\gamma}| &  \textrm{with}\,\,\, \alpha \in
(0,\infty),\,\,\, \beta\in(0,\infty),\,\,\,\gamma\in(0,1] \,\,\,\textrm{and}\,\,\,\lambda =
\alpha\beta+\gamma \in(0,1], \nonumber \\ 
K(x,y) & =   (x+y)^{\lambda} e^{-\beta(x+y)^{-\alpha}} & \textrm{with}\,\,\, \alpha \in
(0,\infty),\,\,\, \beta\in(0,\infty), \,\,\,\textrm{and}\,\,\,\lambda \in(0,1].\nonumber
\end{align} 
On the other hand, the following fragmentation kernels satisfy Hypothesis \ref{HypSto}.
 
$F(x) \equiv  1, $

$F(x) = x^\alpha,\,\textrm{with } \alpha >0.$

\begin{remark}\label{on_beta} 
\begin{enumerate}[i)]
\item The property (\ref{Hyp1_Beta}) means that there is no gain of mass due to the dislocation of a particle. Nevertheless, it does not exclude a loss of mass due to the dislocation of the particles.

\item Note that under (\ref{Hyp1_Beta}) we have $\sum_{k\geq 1} \theta_k - 1\leq 0$ $\beta$-a.e., and since $\theta_k \in [0,1)$ for all $k\geq 1$, $\theta_k \leq \theta_k^{\lambda}$, we have
\begin{equation}\label{inebeta}
\left\lbrace
\begin{array}{ll}
1-\theta_1^{\lambda} \leq 1-\theta_1 \leq (1-\theta_1)^{\lambda},\,\,   \beta-a.e.,\\[5mm]
\sum_{k \geq 1} \theta_k^{\lambda} - 1 = \sum_{k \geq 2} \theta_k^{\lambda} - (1 - \theta_1^{\lambda}) \leq \sum_{k \geq 2} \theta_k^{\lambda},\,\, \beta-a.e.
\end{array} \right.
\end{equation} 
implying the following bounds:
\begin{equation}\label{Clambdabounds}\left\lbrace
\begin{array}{c}
\displaystyle\int_{\Theta} (1-\theta_1) \beta(d\theta)  \leq  C_{\beta}^{\lambda}, \,\,\,\displaystyle \int_{\Theta} \left[\sum_{k\geq 2} \theta_k^{\lambda} + (1-\theta_1^{\lambda}) \right] \beta(d\theta)   \leq  C_{\beta}^{\lambda}, \\[5mm]
\displaystyle \int_{\Theta} \left( \sum_{k \geq 1} \theta_k^{\lambda} -1 \right)^+\beta(d\theta)    \leq C_{\beta}^{\lambda}.
\end{array} \right.
%\int_{\Theta} (1-\theta_1)^{\lambda} \beta(d\theta)  \leq  C_{\beta}^{\lambda},
\end{equation}
We point out that  $\int_{\Theta}\left| \sum_{k \geq 1} \theta_k^{\lambda} -1 \right|\beta(d\theta)    \leq  2 C_{\beta}^{\lambda}$ but when the term $\sum_{k \geq 1} \theta_k^{\lambda} - 1$ is negative our calculations can be realized in a simpler manner. We will thus use the positive bound given in the last inequality.
\end{enumerate}
\end{remark}

\noindent Within the whole paper, we will use the convention that, when dealing with sequences in $\ell_{\lambda}$,
\begin{eqnarray*}
K(x,0) = 0 & \textrm{\,\,for all \,\,} x\in[0,\infty),\\
F(0) = 0.  & 
\end{eqnarray*}
We will always use this convention, even in the case where, e.g.,  $K\equiv 1$ on $(0,\infty)\times(0,\infty)$ and $F\equiv 1$ on $(0,\infty)$. Actually, $0$ is a symbol used to refer to a particle that does not exist. For $\theta \in \Theta$ and $x \in (0,\infty)$ we will write $\theta\cdot x$ to say that the particle of mass $x$ of the system splits into $\theta_1 x, \theta_2 x,\ldots$\medskip

\noindent Furthermore, we will refer to the property of ``local boundedness" of the coagulation and fragmentation kernels in the sense that for all $a>0$, $\sup_{(0,a]^2} K(x,y) <\infty$ and $\sup_{(0,a]} F(x) <\infty$.\medskip

Considering  $m\in \ell_{\lambda}$, the dynamics of the process is as follows. A pair of particles $m_i$ and $m_j$ coalesce with rate given by $K(m_i,m_j)$ and this is described by the map $c_{ij}: \ell_{\lambda} \rightarrow \ell_{\lambda}$ (see below). A particle $m_i$ fragments following the dislocation configuration $\theta \in \Theta$  with rate given by $F(m_i)\beta(d\theta)$ and this is described by the map $f_{i\theta}: \ell_{\lambda} \rightarrow \ell_{\lambda}$, with
\begin{equation}\label{c_and_f}
\begin{array}{lcl}
c_{ij}(m) & = & \textrm{reorder} (m_1,\ldots,m_{i-1},m_{i}+m_{j}, m_{i+1},\ldots,m_{j-1},m_{j+1},\ldots), \\
f_{i\theta}(m) & = & \textrm{reorder} (m_1,\ldots, m_{i-1},\theta\cdot m_i, m_{i+1},\ldots), 
\end{array}
\end{equation}
the reordering being in the decreasing order.\medskip

\noindent \textbf{Distances on $S^{\downarrow}$}\medskip

We endow $S^{\downarrow}$ with the pointwise convergence topology, which can be metrized by the distance
\begin{equation}\label{distance1}
d(m,\tilde{m}) = \sum_{k\geq 1} 2^{-k}|m_k - \tilde{m}_k|.
\end{equation}

\noindent Also, for $\lambda \in (0,1]$ and $m,\,\tilde{m} \in \ell_{\lambda}$, we recall that since the masses are decreasingly ordered, from \cite[Lemma 3.1.]{Sto-Coal} we have the equality
\begin{equation}\label{def_distance}
\delta_{\lambda}(m,\tilde{m}) = \inf_{\pi,\sigma\in Perm(\mathbb N)} \sum_{i\geq 1}|m_{\pi(i)}^{\lambda} - \tilde m_{\sigma(i)}^{\lambda}| =  \sum_{k\geq 1} |m_k^{\lambda} - \tilde{m}_k^{\lambda}|,
\end{equation}
In this paper we will use the second equality.\medskip

\noindent \textbf{Infinitesimal generator $\mathcal L_{K,F}^{\beta}$}\medskip

\noindent Considering some coagulation and fragmentation kernels $K$ and $F$ and a measure $\beta$. We define the infinitesimal generator $\mathcal L_{K,F}^{\beta}$ for any $\Phi: \ell_{\lambda} \rightarrow \mathbb R$ sufficiently regular and for any $m \in \ell_{\lambda}$ by
\begin{equation}\label{generator}
\mathcal L_{K,F}^{\beta} \Phi(m) = \sum_{1\leq i < j < \infty} K(m_i,m_j) \left[ \Phi\left(c_{ij}(m)\right) - \Phi(m) \right] +  \sum_{i\geq 1} F(m_i) \int_{\Theta} \left[ \Phi\left(f_{i\theta}(m)\right) - \Phi(m) \right] \beta(d\theta).
\end{equation}

\section{Results}\label{Result}
\setcounter{equation}{0}
We first define the finite coalescence - fragmentation process. In order to properly define this process we need to add two properties to the measure $\beta$. Namely, the measure of $\Theta$ must be finite and the number of fragments at each fragmentation must be bounded:
\begin{equation}\label{Hypfinite_beta}\left\lbrace
\begin{array}{rccr}
\beta(\Theta) & < & \infty, & \\
\beta(\Theta\setminus \Theta_k) & = & 0 & \textrm{ for some } k\in\mathbb N,
\end{array} \right.
\end{equation}
where
\begin{equation}
\Theta_k = \left\lbrace \theta = (\theta_n)_{n \geq 1} \in \Theta :\, \theta_{k+1} = \theta_{k+2} = \cdots = 0\right\rbrace. \nonumber
\end{equation}

\begin{proposition}[Finite Coalescence - Fragmentation processes]\label{Def_finiteproc}
%Consider $\lambda \in (0,1]$, $\alpha\geq 0$ and  $m\in\ell_{0+}$. Assume that the coagulation kernel $K$, the fragmentation kernel $F$ and a measure $\beta$ satisfy Hypotheses \ref{HypSto}. Furthermore, suppose that $\beta$ satisfies (\ref{Hypfinite_beta}). \medskip

Consider $\lambda \in (0,1]$, $\alpha\geq 0$ and  $m\in\ell_{0+}$. Assume that a coagulation kernel $K$ bounded on compact subsets on $[0,\infty)^2$, a fragmentation kernel $F$ bounded on compact subsets of $[0,\infty)$ and a measure $\beta$ satisfy Hypotheses \ref{HypSto}. Furthermore, suppose that $\beta$ satisfies (\ref{Hypfinite_beta}). \medskip

Then, there exists a unique (in law) strong Markov process $(M(m,t))_{t\geq 0}$ starting at $M(m,0) = m$ and with infinitesimal generator  $\mathcal L_{K,F}^{\beta}$.
\end{proposition}

We point out that in order to prove existence and uniqueness of the Finite Coalescence - Fragmentation process, kernels $K$ and $F$ do not need to satisfy the continuity conditions (\ref{Hyp_CoagK}) and (\ref{Hyp_FragK}), respectively. The proof is based on the existence and uniqueness of its Poissonian representation Proposition \ref{Finite_PoissonProc}. for which the jump intensity remains bounded on finite time-intervals Lemma \ref{finite_intensities}. \medskip

We wish to extend this process to the case where the initial condition consists of an infinite number of particles and for more general fragmentation measures $\beta$ under some additional continuity conditions for the kernels. For this, we will build a particular sequence of finite coalescence - fragmentation processes, the result will be obtained by passing to the limit. \medskip

We introduce the following notation that will be useful when working with finite processes. We consider a measure $\beta$ satisfying Hypotheses \ref{HypSto}., $n\in \mathbb N$ and the set $\Theta(n)$ defined by $ \Theta(n) = \left\lbrace \theta \in \Theta : \theta_1 \leq 1-\frac{1}{n} \right\rbrace$, we consider also the projector 
\begin{equation}\label{projector}
\begin{array}{llcl}
\psi_n: & \Theta & \rightarrow & \Theta_n \\
		& \theta & \mapsto & \psi_n(\theta) = (\theta_1,\ldots,\theta_n,0,\ldots),
\end{array}
\end{equation}
and we put
\begin{equation} \label{beta_n}
\beta_n  = \mathds 1_{\theta\in\Theta(n)}\beta \circ \psi_n^{-1}.
\end{equation}
The measure $\beta_n$ can be seen as the restriction of $\beta$ to the projection of $\Theta(n)$ onto $\Theta_n$. Note that $\Theta(n) \subset \Theta(n+1)$ and that since we have excluded the degenerated cases $\theta_1=1$ we have $\bigcup_n \Theta(n) = \Theta$.

\begin{lemma}[Definition.- The finite process $M^n(m,t)$] \label{lemmaFinite}Consider $\lambda \in (0,1]$, $\alpha\geq0$ and  $m\in\ell_{0+}$. Assume that the coagulation kernel $K$, the fragmentation kernel $F$ and the measure $\beta$ satisfy Hypotheses \ref{HypSto}. Furthermore, recall $\beta_n$ as defined by (\ref{beta_n}).\medskip

Then, there exists a unique (in law) strong Markov process $(M^n(m,t))_{t\geq 0}$ starting at $m$ and with infinitesimal generator $\mathcal L_{K,F}^{\beta_n}$.
\end{lemma}
This lemma is straightforward, it suffices to note that $\beta_n$ satisfies (\ref{Hypfinite_beta}),  that the kernels $K$ and $F$ are locally bounded since they satisfy respectively (\ref{Hyp_CoagK}) and (\ref{Hyp_FragK}) and to use Proposition \ref{Def_finiteproc}. Indeed, recall (\ref{Clambdabounds}), for $n\geq 1$
\begin{equation}
\beta_n(\Theta) = \int_{\Theta} \mathds 1_{\{1-[\psi_n(\theta)]_1 \geq \frac{1}{n}\}} \, \beta(d\theta) \leq n \int_{\Theta} (1-\theta_1) \, \beta(d\theta) \leq n\, C^{\lambda}_{\beta} < \infty.\nonumber
\end{equation}

We have chosen an explicit sequence of measures $(\beta_n)_{n\geq1}$ because it will be easier to manipulate when coupling two coalescence-fragmentation processes. Nevertheless, more generally, taking any sequence of measures $\beta_n$ satisfying (\ref{Hypfinite_beta}) and converging towards $\beta$ in a suitable sense as $n$ tends to infinity should provide the same result.\medskip

Our main result concerning stochastic Coalescence-Fragmentation processes is the following.

\begin{theorem}\label{Result_Theo}
Consider $\lambda \in (0,1]$, $\alpha\geq0$. Assume that the coagulation $K$ and the fragmentation $F$ kernels and a measure $\beta$ satisfy Hypotheses \ref{HypSto}. Endow $\ell_{\lambda}$ with the distance $\delta_{\lambda}$.
\begin{enumerate}[i)]
 \item \label{ResultTheo_i} For any $m\in\ell_{\lambda}$, there exists a (necessarily unique in law) strong Markov process $(M(m,t))_{t\geq 0} \in \mathbb D\left([0,\infty),\ell_{\lambda}\right)$ satisfying the following property.\medskip
 
\noindent For any sequence $m^n\in \ell_{0+}$ such that \,$\lim_{n\rightarrow \infty}\delta_{\lambda}(m^n,m)$ $=0$, the sequence $(M^n(m^n,t))_{t\geq0}$ defined in Lemma \ref{lemmaFinite}, converges in law, in $\mathbb D\left([0,\infty),\ell_{\lambda}\right)$, to  $(M(m,t))_{t\geq 0}$.
\item \label{ResultTheo_ii} The obtained process is Feller in the sense that for all $t\geq 0$, the map $m\mapsto Law\left(M(m,t)\right)$ is continuous from $\ell_{\lambda}$ into $\mathcal P(\ell_{\lambda})$ (endowed with the distance $\delta_{\lambda}$).
\item \label{ResultTheo_iii} Recall the expression  (\ref{distance1}) of the distance $d$. For all bounded application  $\Phi:\ell_{\lambda} \rightarrow \mathbb R$ satisfying $|\Phi(m) - \Phi(\tilde m)| \leq a \,d(m,\tilde m)$ for some $a>0$, the process 
\begin{eqnarray*}
\Phi\left(M(m,t)\right) - \Phi\left(m\right) - \int_0^t \mathcal L_{K,F}^{\beta}\left(M(m,s)\right) ds
\end{eqnarray*}
is a local martingale.
\end{enumerate}
\end{theorem}

This result extends those of Fournier \cite{Sto-Coal} concerning solely coalescence and Bertoin \cite{Bertoin,Bertoin2} concerning only fragmentation.  We point out that in \cite{Bertoin} is not assumed $C_{\beta}^{\lambda}< \infty$ but only $\int_{\Theta}(1-\theta_1)\beta(d\theta) < \infty$. However, we believe that in the presence of coalescence our hypotheses on $\beta$ are optimal. \medskip

Theorem \ref{Result_Theo}. will be proved in two steps, the first step consists in proving existence and uniqueness of the Finite Coalescence-Fragmentation process, finite in the sense that it is composed by a finite number of particles for all $t\geq0$. Next, we will use a sequence of finite processes to build a process as its limit, where the system is composed by an infinite number of particles. The construction of such processes uses a Poissonian representation which is introduced in the next section.

\section{A Poisson-driven S.D.E.}\label{SDEIntro}
\setcounter{equation}{0}
We now introduce a representation of the stochastic processes of coagulation - fragmentation in terms of Poisson measures, in order to couple two of these processes with different initial data.
 
\begin{definition}\label{Def_Proc}  
Assume that a coagulation kernel $K$, a fragmentation kernel $F$ and a measure $\beta$ satisfy Hypotheses \ref{HypSto}.
\begin{enumerate}[a)]
\item For the coagulation, we consider a Poisson measure $\mathcal N(dt,d(i,j),dz)$ on $[0,\infty) \times \{(i,j) \in\mathbb N^2, i<j\} \times [0,\infty)$ with intensity measure  $dt$ $\left[\sum_{k<l} \delta_{(k,l)}(d(i,j)) \right]$ $dz$, and denote by $(\mathcal F_t)_{t\geq 0}$ the associated canonical filtration.
\item For the fragmentation, we consider  $\mathcal M(dt,di,d\theta,dz)$ a Poisson measure on $[0,\infty) \times \mathbb N \times \Theta \times [0,\infty)$ with intensity measure  $dt$ $\left(\sum_{k \geq 1} \delta_{k}(di) \right)$ $\beta(d\theta)$ $dz$, and denote by $(\mathcal G_t)_{t\geq 0}$ the associated canonical filtration. $\mathcal M$ is independent of $\mathcal N$.
\end{enumerate}
Finally, we consider $m \in \mathcal \ell_{\lambda}$. A c\`adl\`ag  $(\mathcal H_t)_{t\geq 0} = \left(\sigma(\mathcal F_t , \mathcal G_t)\right)_{t\geq 0}$-adapted process $(M(m,t))_{t\geq 0}$ is said to be a solution to $SDE(K,F,m,\mathcal N,\mathcal M)$ if it belongs a.s. to $\mathbb D\left([0,\infty),\ell_{\lambda}\right)$ and if for all $t\geq0$, a.s. 
\begin{eqnarray}\label{def_sto-process}
M(m,t) & = & m + \int_0^t \int_{i<j} \int_0^{\infty} \left[ c_{ij}\left(M(m,s-)\right) - M(m,s-)\right] \mathds 1_{\lbrace z \leq K(M_i(m,s-),M_j(m,s-))\rbrace}\nonumber\\
& & \hspace{9cm}\mathcal  N(ds,d(i,j),dz) \nonumber\\
&&+ \int_0^t \int_{i} \int_{\Theta} \int_0^{\infty} \left[ f_{i\theta}\left(M(m,s-)\right) - M(m,s-)\right] \mathds 1_{ \lbrace z \leq F(M_i(m,s-))\rbrace}\nonumber\\
& & \hspace{9cm}\mathcal M(ds,di,d\theta,dz).
\end{eqnarray}
\end{definition}
Remark that due to the independence of the Poisson measures only a coagulation or a fragmentation event occurs at each instant $t$.\medskip

\noindent We begin by checking that the integrals in (\ref{def_sto-process}) always make sense. 

\begin{lemma}\label{lemma_wellDef}
Let $\lambda\in(0,1]$ and $\alpha\geq 0$,  the coagulation kernel $K$ be bounded on compact subsets on $[0,\infty)^2$, the fragmentation kernel $F$ be bounded on compact subsets of $[0,\infty)$, and the $\beta$ and the Poisson measures $\mathcal N$ and $\mathcal M$ as in Definition \ref{Def_Proc}. For any $\left(\mathcal H_t\right)_{t\geq 0}$-adapted process $\left(M(t)\right)_{t\geq 0}$ belonging a.s. to $\mathbb D\left([0,\infty),\ell_{\lambda}\right)$, a.s.
\begin{eqnarray*}
I_1 & = & \int_0^t \int_{i<j} \int_0^{\infty} \left[ c_{ij}\left(M(s-)\right) - M(s-)\right] \mathds 1_{\lbrace z \leq K(M_i(s-),M_j(s-))\rbrace} \mathcal N(ds,d(i,j),dz),\\
I_2 & = & \int_0^t \int_{i} \int_{\Theta} \int_0^{\infty} \left[ f_{i\theta}\left(M(s-)\right) - M(s-)\right] \mathds 1_{ \lbrace z \leq F(M_i(s-))\rbrace}\mathcal M(ds,di,d\theta,dz),
\end{eqnarray*}
are well-defined and finite for all $t\geq 0$.
\end{lemma}

\begin{proof}%[Proof of Lemma \ref{lemma_wellDef}]
The processes in the integral being c\`adl\`ag and adapted, it suffices to check the compensators are a.s. finite. We have to show that a.s., for all $k\geq 1$, all $t\geq 0$,
\begin{eqnarray*}
C_k(t) & = & \int_0^t ds \sum_{i<j} K(M_i(s),M_j(s)) |[c_{ij}(M(s))]_k - M_k(s)| \\
& & + \int_0^t ds\int_{\Theta}\beta(d\theta) \sum_{i\geq 1} F(M_i(s)) |[f_{i\theta}(M(s))]_k - M_k(s)| < \infty.
\end{eqnarray*}
Note first that for all $s \in[0,t]$, $\sup_i M_i(s) \leq \sup_{[0,t]} \|M(s)\|_1 \leq \sup_{[0,t]} \|M(s)\|^{1/\lambda}_{\lambda}=: a_t <\infty$ a.s. since $M$ belongs a.s. to $\mathbb D\left([0,\infty),\ell_{\lambda}\right)$. Next, let 
\begin{equation}\label{Koline}
\overline K_t = \sup_{(x,y) \in [0,a_t]^2}K(x,y) \hspace{1cm} \textrm{and} \hspace{1cm} \overline F_t = \sup_{x \in [0,a_t]}F(x),
\end{equation} 
which are $a.s.$ finite since $K$ and $F$ are bounded on every compact in $[0,\infty)^2$ and $[0,\infty)$, respectively. Then using (\ref{d1c}) and (\ref{d1f}) with (\ref{inebeta}) and (\ref{Clambdabounds}), we write:
\begin{eqnarray*}
\sum_{k\geq 1} 2^{-k} C_k(t) & = & \int_0^t ds \sum_{i<j} K(M_i(s),M_j(s))\, d\left(c_{ij}(M(s)), M(s)\right) \\
& & + \int_0^t ds\int_{\Theta}\beta(d\theta) \sum_{i\geq 1} F(M_i(s))\, d\left(f_{i\theta}(M(s)), M(s)\right)  \\
& \leq & \overline K_t \int_0^t ds \sum_{i<j}  \frac{3}{2}2^{-i}M_j(s) +  C^{\lambda}_{\beta} \overline{F}_t \int_0^t ds  \sum_{i\geq1} 2^{-i} M_i(s)\\ 
& \leq &  \left( \frac{3}{2}\overline K_t + C^{\lambda}_{\beta} \overline{F}_t\right) \int_0^t \|M(s)\|_1 ds \\
&\leq &  t \left( \frac{3}{2}\overline K_t + C^{\lambda}_{\beta} \overline{F}_t\right) \sup_{[0,t]} \|M(s)\|_{\lambda}^{1/\lambda}  \,\,\,< \,\,\, \infty.
\end{eqnarray*}
\end{proof}

\subsection{Existence and uniqueness for \textit{SDE}: finite case}\label{EUFcase}
The aim of this paragraph is to prove Proposition \ref{Def_finiteproc}. This proposition is a consequence of Proposition \ref{Finite_PoissonProc}. bellow. We will first prove existence and uniqueness of the Finite Coalescence - Fragmentation processes satisfiying $(SDE)$ and then some fundamental inequalities. 

\begin{proposition}\label{Finite_PoissonProc}
Let $m\in\ell_{0+}$. Consider  a coagulation kernel $K$ bounded on compact subsets of $[0,\infty)^2$, a fragmentation kernel $F$ bounded on compact subsets of $[0,\infty)$ and a measure $\beta$ and the Poisson measures $\mathcal N$ and $\mathcal M$ as  in Definition \ref{Def_Proc}, suppose furthermore that $\beta$ satisfies (\ref{Hypfinite_beta}).\medskip

\noindent Then there exists a unique process $(M(m,t))_{t\geq 0}$ which solves $SDE(K,F,m,\mathcal N,\mathcal M)$. This process is a finite Coalescence-Fragmentation process in the sense of Proposition \ref{Def_finiteproc}.
\end{proposition}

We recall that in order to prove Proposition \ref{Finite_PoissonProc}. the kernels $K$ and $F$ do not need to satisfy the continuity conditions (\ref{Hyp_CoagK}) and (\ref{Hyp_FragK}), we need only to assume local boundness to prove that the jump intensity is bounded on finite time-intervals. The continuity conditions on both kernels are needed, in general, when considering an infinite number of particles in the system and in particular, to control the distance $\delta_{\lambda}$ between two solutions to \textit{SDE} Proposition \ref{Result_Prop}. \textit{\ref{Prop_C01_ii})} below.

\subsubsection{A Gronwall type inequality}
We will also check a fundamental inequality, which shows that the distance between two coagulation-fragmentation processes introduced in Proposition \ref{Finite_PoissonProc}. cannot increase excessively while their moments of order $\lambda$ remain finite. For this, we need to consider the additional continuity conditions (\ref{Hyp_CoagK}) and (\ref{Hyp_FragK}).
\begin{proposition}\label{Result_Prop}
Let $\lambda\in(0,1]$, $\alpha\geq 0$ and $m, \tilde m\in\ell_{0+}$. Consider $K$, $F$, $\beta$ and the Poisson measures $\mathcal N$ and $\mathcal M$ as in Definition \ref{Def_Proc}, we furthermore suppose that $\beta$ satisfies (\ref{Hypfinite_beta}). Consider the unique solutions $M(m,t)$ and $M(\tilde m,t)$ to $SDE(K,F,m,\mathcal N,\mathcal M)$ and $SDE(K,F,\tilde m,\mathcal N,\mathcal M)$ constructed in Proposition \ref{Finite_PoissonProc}. and recall $C^{\lambda}_{\beta}$ (\ref{Hyp2_Beta}).
\begin{enumerate}[i)]
\item \label{Prop_C01_i} The map $t\mapsto \|M(m,t)\|_1$ is a.s. non-increasing. Futhermore, for all $t\geq0$ 
\begin{equation}
\mathbb E\left[\sup_{s\in[0,t]}\|M(m,s)\|_{\lambda}\right] \leq \|m\|_{\lambda}\, e^{\overline{F}_m C^{\lambda}_{\beta}\,\, t},\nonumber
\end{equation}
where $\overline{F}_m = \sup_{[0,\|m\|_1]}F(x)$.
\item \label{Prop_C01_ii} We define, for all $x>0$, the stopping time $\tau(m,x) = \inf\{t\geq0, \|M(m,t)\|_{\lambda} \geq x\}$. Then for all $t\geq0$ and all $x>0$,  
\begin{equation}
\mathbb E\left[\sup_{s\in[0,t\wedge \tau(m,x) \wedge \tau(\tilde m,x)]}  \delta_{\lambda}\left(M(m,s), M(\tilde m,s) \right) \right] \leq \delta_{\lambda}\left(m, \tilde{m} \right) e^{C (x+ 1)\,t}. \nonumber
\end{equation}
where $C$ is a positive constant depending on $K$, $F$, $C^{\lambda}_{\beta}$, $\|m\|_{1}$ and $\|\tilde m\|_{1}$.
\end{enumerate}
\end{proposition}

This proposition will be useful to construct a process in the sense of Definition \ref{Def_Proc}. as the limit of a sequence of approximations. It will provide some important uniform bounds not depending on the approximations but only on the initial conditions and $C_{\beta}^{\lambda}$.\medskip

\subsubsection{Proofs}\label{Proof}
In this section we provide proofs to propositions \ref{Finite_PoissonProc}., \ref{Def_finiteproc}. and \ref{Result_Prop}.

\begin{proof}[Proof of Proposition \ref{Finite_PoissonProc}]

This proposition will be proved considering that in such a system the number of particles remains finite. We will conclude using the fact that the total rate of jumps of the system is bounded by the number of particles. 

\begin{lemma}\label{finite_intensities} 
Let $m\in\ell_{0+}$, consider  a coagulation kernel $K$ bounded on compact subsets on $[0,\infty)^2$, a fragmentation kernel $F$ bounded on compact subsets of $[0,\infty)$, and $\beta$ and the Poisson measures $\mathcal N$ and $\mathcal M$ as in Definition \ref{Def_Proc}. and assume that $\beta$ satisfies (\ref{Hypfinite_beta}). Assume that there exists $(M(m,t))_{t\geq 0}$ solution  to $SDE(K,F,m,\mathcal N,\mathcal M)$.
\begin{enumerate}[i)]
%\item\label{0lemma} $\|M(m,t)\|_1 \leq \|m\|_1, a.s.$, for all $t\geq0$.
\item\label{ilemma} The number of particles in the system remains $a.s.$ bounded on finite time-intervals,
\begin{equation}
\sup_{s\in[0,t]} N_s < \infty,\,\,a.s. \,\,\textrm{for all } \,\,t\geq 0,\nonumber
\end{equation}
where $N_t = card \{M_i(m,t) : M_i(m,t) > 0  \} = \sum_{i\geq 1} \mathds 1_{\{M_i(m,t) >0 \}}$.
\item\label{iilemma}  The coalescence and fragmentation jump rates of the process $(M(m,t))_{t\geq 0}$ are $a.s.$ bounded on finite time-intervals, this is
\begin{equation}
\sup_{s\in[0,t]}\left(\rho_c(s)+\rho_f(s)\right)  < \infty, \,\,a.s. \,\,\textrm{for all } \,\,t\geq 0, \nonumber%\\
%\sup_{s\in[0,t]}\rho_f(s) & < & \infty, \,\,a.s. \,\,\textrm{for all } \,\,t\geq 0,\nonumber
\end{equation} 
where $\rho_c(t) := \sum_{i<j} K(M_i(m,t),M_j(m,t))$ and $\rho_f(t) := \beta(\Theta) \sum_{i\geq 1} F(M_i(m,t))$.
\end{enumerate}
\end{lemma}

\begin{proof} 
First, denoting $\overline K_m := \sup_{[0,\|m\|_1]^2}K(x,y)$ and $\overline F_m := \sup_{[0,\|m\|_1]}F(x)$, note that we have $\rho_c(0) \leq \overline K_m N_0^2$ and $\rho_f(0)\leq \beta(\Theta) \overline F_m N_0$, which shows that the initial total jump intensity of the system is finite and that the first jump time is strictly positive $T_1>0$. We can thus prove by recurrence that there exists a sequence $0<T_1<\ldots<T_j<\ldots<T_{\infty}$ of jumping times with $T_{\infty} = \lim_{j\rightarrow\infty}T_j$. We now prove that $T_{\infty} = \infty$.\medskip

Let $L^f(t) := card\{j\geq1 : T_j\leq t\,\,\textrm{and}\,\,T_j\,\,\textrm{is a jump of}\,\, M\}$ be the number of fragmentations in the system until the instant $t\geq 0$. Recall that the measure $\beta$ satisfies (\ref{Hypfinite_beta}), since $k$ is the maximum number of fragments, it is easy to see that
\begin{equation}
N_t\leq N_0 + (k-1)L^f(t) <\infty\,\,a.s.,\,\,\textrm{for all}\, t<T_{\infty}.\nonumber
\end{equation}
Applying now (\ref{generator}) with $\Psi(m) = \sum_{n \geq 1} m_n$ and since that $\Psi(c_{ij}(m))-\Psi(m) = 0$ and $\Psi(f_{i\theta}(m))-\Psi(m) = m_i\left(\sum_{i=1}^k \theta_i - 1\right)  \leq 0,\,\beta-a.e.$, we obtain
\begin{equation}
\sup_{s\in[0,t]}\|M(m,s)\|_1  \leq \|m\|_1, a.s.,\,\,\textrm{for all}\,\, t<T_{\infty},\nonumber
\end{equation} 
which implies, $a.s.$ for all $t<T_{\infty}$,
\begin{equation}\label{intensityTinf}
\left\lbrace
\begin{array}{lcl}
\rho_c(t) & \leq & \overline K_m N_{t-}^2,\\
\rho_f(t) & \leq & \beta(\Theta) \overline F_m N_{t-}.
\end{array} \right.
\end{equation}

Next, define $\Phi(m) = \sum_{n\geq 1} \mathds 1_{\{m_n > 0\}}$, recall (\ref{generator}) and use $\Phi(c_{ij}(m))-\Phi(m) \leq 0$, to obtain  
\begin{eqnarray*}
\mathcal L_{K,F}^{\beta} \Phi(m) & \leq & %\sum_{1\leq i < j < \infty}  K(m_i,m_j)\left[ \Phi\left(c_{ij}(m)\right) - \Phi(m) \right] \\
%&& \hspace{3cm}+  \, 
\sum_{i\geq 1} \int_{\Theta} F(m_i) \left[ \Phi\left(f_{i\theta}(m)\right) - \Phi(m) \right] \beta(d\theta)\\\\
&\leq & \overline{F}_m\,  \sum_{i\geq 1} \int_{\Theta}  \left[\sum_{n \geq 1} \mathds 1_{\{\theta_n m_i > 0\}}-\mathds 1_{\{m_i > 0\}} \right] \beta(d\theta)\\
& \leq & (k-1)\,\overline{F}_m\, \beta(\Theta) \Phi(m),
\end{eqnarray*}
we used $\theta_j m_i = 0$ for all $j\geq k+1$.\medskip

\noindent Hence, we have for all $t\geq0$,
\begin{eqnarray*}
\mathbb E\left[\sup_{s\in[0,t\wedge T_{\infty})}N_{s} \right] & \leq & N_0 + (k-1)\,\overline{F}_m\, \beta(\Theta) \mathbb E\left[\int_0^{t\wedge T_{\infty}} N_{s-}ds\right]\\
&\leq & N_0 + (k-1)\,\overline{F}_m\, \beta(\Theta) \int_0^t\mathbb  E\left[\sup_{u\in[0,s\wedge T_{\infty})} N_{u}\right]du.
\end{eqnarray*}
We use the Gronwall Lemma to obtain
\begin{eqnarray*}
\mathbb E\left[\sup_{s\in[0,t\wedge T_{\infty})} N_{s} \right] & \leq &  N_0\, e^{(k-1)\,\overline{F}_m\,\beta(\Theta) t},
\end{eqnarray*}
for all $t\geq0$. We thus deduce,
\begin{equation}\label{Nfinias}
\sup_{s\in[0,t\wedge T_{\infty})} N_{s} < \infty,\,\,a.s.,
\end{equation}
for all $t\geq0$. \medskip

Suppose now that $T_{\infty}<\infty$, then from (\ref{Nfinias}) we deduce that $\sup_{t\in[0,T_{\infty})} N_{t} < \infty,\,\,a.s.$. which means that, using  (\ref{intensityTinf}), $\sup_{t\in[0,T_{\infty})} (\rho_c(t)+\rho_f(t)) < \infty,\,\,a.s$. This is in contradiction with $T_{\infty}<\infty$ since the total jump intensity necessarily explodes to infinity on $T_{\infty}$ when $T_{\infty}<\infty$.\medskip

We deduce that,  
\begin{eqnarray*}
\mathbb E\left[\sup_{s\in[0,t]} N_{s} \right] & \leq &  N_0\, e^{(k-1)\,\overline{F}_m\,\beta(\Theta) t},
\end{eqnarray*}
for all $t\geq0$, and \textit{\ref{ilemma})} readily follows. Finally, \textit{\ref{iilemma})} follows easily from \textit{\ref{ilemma})} and (\ref{intensityTinf}). \medskip

\noindent This ends the proof of Lemma \ref{finite_intensities}. 
\end{proof}
From Lemma \ref{finite_intensities}. we deduce that the total rate of jumps of the system is uniformly bounded. Thus, pathwise existence and uniqueness holds for $(M(m,t))_{t\geq 0}$ solution to $SDE(K,F,m,\mathcal N,\mathcal M)$.\medskip

This ends the proof of Proposition \ref{Finite_PoissonProc}.
\end{proof}

\begin{proof}[Proof of Proposition \ref{Def_finiteproc}]
Let $\lambda\in(0,1]$, $\alpha\geq 0$ and $m \in \ell_{0+}$, and consider $K$, $F$, $\beta$ and the Poisson measures $\mathcal N$ and $\mathcal M$ as in Proposition \ref{Def_finiteproc}.\medskip

Consider the process $(M(m,t))_{t\geq 0}$, the unique solution to $SDE(K,F,m,\mathcal N,\mathcal M)$ built in Proposition \ref{Finite_PoissonProc}. The system $(M(m,t))_{t\geq 0}$ is a strong Markov process in continuous time with infinitesimal generator  $\mathcal L_{K,F}^{\beta}$ and Proposition \ref{Def_finiteproc}. follows.
\end{proof}

\begin{proof}[Proof of Proposition \ref{Result_Prop}]

Let $\lambda\in(0,1]$, $\alpha\geq 0$ and $m \in \ell_{0+}$, and consider $(M(m,t))_{t\geq0}$ the solution to $SDE(K,F,m,\mathcal N,\mathcal M)$ constructed in Proposition \ref{Finite_PoissonProc}. We begin studying the behavior of the moments of this solution.  \medskip

First, we will see that under our assumptions the total mass $\|\cdot\|_1$ does $a.s.$ not increase in time. This property is fundamental in this approach since we will use the bound $\sup_{[0,\|M(m,0)\|_1]}F(x)$, which is finite  whenever $\|M(m,0)\|_{\lambda}$ is. This will allows us to bound lower moments of $M(m,t)$ for $t\geq 0$. \medskip

Next, we will prove that the $\lambda$-moment remains finite in time. Finally, we will show that the distance $\delta_{\lambda}$ between two solutions to (\ref{def_sto-process}) is bounded in time while their $\lambda$-moments remain finite.  \medskip

We point out that in these paragraphs we will use more general estimates for $m\in\ell_{\lambda}$ and $\beta$ satisfying Hypotheses \ref{HypSto}. and not necessarily (\ref{Hypfinite_beta}). This will provide uniform bound when dealing with finite processes.\medskip

\noindent \textbf{Moments Estimates.- }The aim of this paragraph is to prove \textit{Proposition \ref{Result_Prop}. \ref{Prop_C01_i})}.\medskip 

\noindent The solution to $SDE(K,F,m,\mathcal N,\mathcal M)$ will be written $M(t) := M(m,t)$ for simplicity. From Lemma \ref{finite_intensities}. \textit{\ref{ilemma})}, we know that the number of particles in the system is $a.s.$ finite and thus the following sums are obviously well-defined.\medskip 

First, from (\ref{def_sto-process}) we have for $k\geq 1$,

\begin{eqnarray}\label{Proof:Mk}
M_k(t) & = & M_k(0) + \int_0^t \int_{i<j} \int_0^{\infty} \left[ [c_{ij}\left(M(s-)\right)]_k - M_k(s-)\right] \mathds 1_{\lbrace z \leq K(M_i(s-),M_j(s-))\rbrace} \nonumber\\ 
&&\hspace{9.5cm}\mathcal N(ds,d(i,j),dz) \nonumber\\
&&+ \int_0^t \int_{i} \int_{\Theta} \int_0^{\infty} \left[ [f_{i\theta}\left(M(s-)\right)]_k -M(s-)_k\right] \mathds 1_{ \lbrace z \leq F(M_i(s-))\rbrace}\nonumber\\ 
&&\hspace{9.5cm}\mathcal M(ds,di,d\theta,dz),
\end{eqnarray}
and summing on $k$, we deduce
\begin{eqnarray}\label{Proof:M1}
\|M(t)\|_{1} & = & \|m\|_{1} + \int_0^{t} \int_{i<j} \int_0^{\infty} \left[ \|c_{ij}\left(M(s-)\right)\|_{1} - \| M(s-)\|_{1}\right]  \mathds 1_{\lbrace z \leq K(M_i(s-),M_j(s-))\rbrace}  \nonumber \\ &&\hspace{9.5cm}\mathcal N(ds,d(i,j),dz)\nonumber \\
&  & + \int_0^{t} \int_{i} \int_0^{\infty} \left[ \|f_{i\theta}\left(M(s-)\right)\|_{1} - \| M(s-)\|_{1}\right] \mathds 1_{ \lbrace z \leq F(M_i(s-))\rbrace} \nonumber \\ 
&&\hspace{9.5cm}\mathcal M(ds,di,d\theta,dz).
\end{eqnarray}
Note that, clearly $\|c_{ij}\left(m\right)\|_{1} = \|m\|_{1}$  and $\|f_{i\theta}\left(m\right)\|_{1} = \|m\|_{1} + m_i \left(\sum_{k\geq 1} \theta_k -1\right) \leq \|m\|_{1}$ for all $m\in \ell_\lambda$, since $\sum_{k\geq 1} \theta_k \leq 1$ $\beta$-a.e. Then, 
\begin{equation}%\label{masseloss}
\sup_{[0,t]}\|M(s)\|_{1}  \leq  \|m\|_{1}, \,a.s.\,\,\, \forall t\geq 0.\nonumber
\end{equation}
This implies for all $s \in[0,t]$, $\sup_i M_i(s) \leq \sup_{[0,t]} \|M(s)\|_1 \leq \|m\|_1\, a.s$. We set
\begin{equation}\label{Koline2}
\overline K_m = \sup_{(x,y) \in [0,\|m\|_1]^2}K(x,y) \hspace{1cm} \textrm{and} \hspace{1cm} \overline F_m = \sup_{x \in [0,\|m\|_1]}F(x)
\end{equation} which are finite since $K$ and $F$ are bounded on every compact in $[0,\infty)^2$ and $[0,\infty)$ respectively.\medskip

In the same way, from (\ref{def_sto-process}) for $\lambda\in(0,1)$ we have for $k\geq 1$,
\begin{eqnarray*}
[M_k(t)]^{\lambda} & = & [M_k(0)]^{\lambda} + \int_0^t \int_{i<j} \int_0^{\infty} \left[ [c_{ij}\left(M(s-)\right)]_k^{\lambda} - [M_k(s-)]^{\lambda}\right] \mathds 1_{\lbrace z \leq K(M_i(s-),M_j(s-))\rbrace} \nonumber\\ 
&&\hspace{9.5cm}\mathcal N(ds,d(i,j),dz) \nonumber\\
&&+ \int_0^t \int_{i} \int_{\Theta} \int_0^{\infty} \left[ [f_{i\theta}\left(M(s-)\right)]_k^{\lambda} - [M(s-)]_k^{\lambda}\right] \mathds 1_{ \lbrace z \leq F(M_i(s-))\rbrace}\nonumber\\ 
&&\hspace{9.5cm}\mathcal M(ds,di,d\theta,dz),
\end{eqnarray*}
and summing on $k$, we deduce
\begin{eqnarray}\label{Proof:Ml}
\|M(t)\|_{\lambda} & = & \|m\|_{\lambda} + \int_0^{t} \int_{i<j} \int_0^{\infty} \left[ \|c_{ij}\left(M(s-)\right)\|_{\lambda} - \| M(s-)\|_{\lambda}\right]  \mathds 1_{\lbrace z \leq K(M_i(s-),M_j(s-))\rbrace}  \nonumber \\ 
&&\hspace{9.5cm}\mathcal N(ds,d(i,j),dz)\nonumber \\
&  & + \int_0^{t} \int_{i} \int_0^{\infty} \left[ \|f_{i\theta}\left(M(s-)\right)\|_{\lambda} - \| M(s-)\|_{\lambda}\right] \mathds 1_{ \lbrace z \leq F(M_i(s-))\rbrace}\nonumber \\ 
&&\hspace{9.5cm}\mathcal M(ds,di,d\theta,dz).
\end{eqnarray}
We take the expectation, use (\ref{C1_01}) and (\ref{F1_01}) with (\ref{Clambdabounds}) and (\ref{Koline2}), to obtain
\begin{eqnarray*}%\label{Moment_lambda}
\mathbb E \left[\sup_{s\in[0,t]} \|M(s)\|_{\lambda}\right]& \leq & \|m\|_{\lambda}  + C_{\beta}^{\lambda} \int_0^t \mathbb E\left[\sum_{i\geq 1} F\left(M_i(s)\right) M_i^{\lambda}(s) \right]ds \\
& \leq & \|m\|_{\lambda} + \overline{F}_m C_{\beta}^{\lambda} \int_0^t \mathbb E\left[  \|M(s)\|_{\lambda} \right]ds.
\end{eqnarray*}
We conclude using the Gronwall Lemma.\medskip

\noindent\textbf{Bound for $\delta_{\lambda}$.-} The aim of this paragraph is to prove \textit{Proposition \ref{Result_Prop}. \ref{Prop_C01_ii})}. For this, we consider for $m,\tilde{m} \in \ell_{\lambda}$ some solutions to  $SDE(K,F,m,\mathcal N,\mathcal M)$ and $SDE(K,F,\tilde m,\mathcal N,\mathcal M)$ which will be written $M(t) := M(m,t)$ and $\tilde M(t) :=  M(\tilde m,t)$ for simplicity. Since $M$ and $\tilde M$ solve (\ref{def_sto-process}) with the same Poisson measures $\mathcal N$ and $\mathcal M$, and since the numbers of particles in the systems are $a.s.$ finite, we have
\begin{equation} \label{decomp_dist}
\delta_{\lambda}(M(t),\tilde M(t)) = \delta_{\lambda}(m,\tilde m) + A_t^c + B_t^c + C_t^c + A_t^f + B_t^f + C_t^f,
\end{equation}
where
\begin{eqnarray*}
A_t^c & = & \int_0^t \int_{i<j} \int_0^{\infty} \left\lbrace \delta_{\lambda}\left(c_{ij}(M(s-)),c_{ij}(\tilde M(s-))\right) - \delta_{\lambda}\left(M(s-),\tilde M(s-)\right)\right\rbrace \\
& & \hspace{3.5cm} \mathds 1_{\left\lbrace z \leq K\left(M_i(s-),M_j(s-)\right) \wedge K\left(\tilde M_i(s-),\tilde M_j(s-)\right) \right\rbrace}\, \mathcal N(ds,d(i,j),dz),
\end{eqnarray*}
\begin{eqnarray*}
B_t^c & = & \int_0^t \int_{i<j} \int_0^{\infty} \left\lbrace \delta_{\lambda}\left(c_{ij}(M(s-)),\tilde M(s-)\right) - \delta_{\lambda}\left(M(s-),\tilde M(s-)\right)\right\rbrace \\
& & \hspace{3.5cm} \mathds 1_{\left\lbrace K\left(\tilde M_i(s-),\tilde M_j(s-)\right)  \leq z \leq K\left(M_i(s-),M_j(s-)\right) \right\rbrace}\, \mathcal N(ds,d(i,j),dz),
\end{eqnarray*}
\begin{eqnarray*}
C_t^c & = & \int_0^t \int_{i<j} \int_0^{\infty} \left\lbrace \delta_{\lambda}\left(M(s-),c_{ij}(\tilde M(s-))\right) - \delta_{\lambda}\left(M(s-),\tilde M(s-)\right)\right\rbrace \\
& & \hspace{3.5cm} \mathds 1_{\left\lbrace K\left( M_i(s-), M_j(s-)\right)  \leq z \leq K\left(\tilde M_i(s-),\tilde M_j(s-)\right) \right\rbrace}\, \mathcal N(ds,d(i,j),dz),
\end{eqnarray*}
\begin{eqnarray*}
A_t^f & = & \int_0^t \int_{i} \int_{\Theta}  \int_0^{\infty} \left\lbrace \delta_{\lambda}\left(f_{i\theta}(M(s-)),f_{i\theta}(\tilde M(s-))\right) - \delta_{\lambda}\left(M(s-),\tilde M(s-)\right)\right\rbrace \\
& & \hspace{6cm} \mathds 1_{\left\lbrace z \leq F\left(M_i(s-)\right) \wedge F\left(\tilde M_i(s-)\right) \right\rbrace}\, \mathcal M(ds,di,d\theta,dz),
\end{eqnarray*}
\begin{eqnarray*}
B_t^f & = & \int_0^t \int_{i} \int_{\Theta} \int_0^{\infty} \left\lbrace \delta_{\lambda}\left(f_{i\theta}(M(s-)),\tilde M(s-)\right) - \delta_{\lambda}\left(M(s-),\tilde M(s-)\right)\right\rbrace \\
& & \hspace{6cm} \mathds 1_{\left\lbrace F\left(\tilde M_i(s-)\right)  \leq z \leq F\left(M_i(s-)\right) \right\rbrace}\,  \mathcal M(ds,di,d\theta,dz),
\end{eqnarray*}
\begin{eqnarray*}
C_t^f & = & \int_0^t \int_{i} \int_{\Theta} \int_0^{\infty} \left\lbrace \delta_{\lambda}\left(M(s-),f_{i\theta}(\tilde M(s-))\right) - \delta_{\lambda}\left(M(s-),\tilde M(s-)\right)\right\rbrace \\
& & \hspace{6cm} \mathds 1_{\left\lbrace F\left( M_i(s-)\right)  \leq z \leq F\left(\tilde M_i(s-)\right) \right\rbrace}\, \mathcal  M(ds,di,d\theta,dz).
\end{eqnarray*}
Note also that
\begin{eqnarray}
\label{inetrian_C}\\
\left| \delta_{\lambda}\left(c_{ij}(M(s-)),\tilde M(s-)\right) - \delta_{\lambda}\left(M(s-),\tilde M(s-)\right) \right| & \leq & \delta_{\lambda}\left(c_{ij}(M(s-)), M(s-)\right) \nonumber  \\
\label{inetrian_F} \\
\left| \delta_{\lambda}\left(f_{i\theta}(M(s-)),\tilde M(s-)\right) - \delta_{\lambda}\left(M(s-),\tilde M(s-)\right) \right| & \leq & \delta_{\lambda}\left(f_{i\theta}(M(s-)), M(s-)\right) \nonumber
\end{eqnarray}

We now search for an upper bound to the expression in (\ref{decomp_dist}). We define, for all $x >0$, the stopping time $\tau(m,x) := \inf \lbrace t\geq 0;\|M(m,t)\|_{\lambda} \geq x \rbrace$. We set $\tau_x =\tau(m,x) \wedge \tau(\tilde m,x)$.\medskip

\noindent Furthermore, since for all $s \in[0,t]$, $\sup_i M_i(s) \leq \sup_{[0,t]} \|M(s)\|_1 \leq \|m\|_1 := a_m\, a.s$, equivalently for $\tilde M$, we put $a_{\tilde m} = \|\tilde m\|_1$. For $a:=a_m \vee a_{\tilde m}$ we set $\kappa_{a}$ and $\mu_{a}$ the constants for which the kernels $K$ and $F$ satisfy (\ref{Hyp_CoagK}) and (\ref{Hyp_FragK}). Finally, we set $\overline F_m$ as in (\ref{Koline2}).  \medskip

\noindent Term $A_t^c$: using (\ref{C3_01}) we deduce that this term is non-positive, we bound it by $0$. \medskip

\noindent Term $B_t^c$: we take the expectation, use (\ref{inetrian_C}), (\ref{C2_01}) and  (\ref{Hyp_CoagK}), to obtain
\begin{eqnarray}\label{Btc_01}
\mathbb E \left[\sup_{s\in[0,t\wedge\tau_x]}B_s^c\right] & \leq & \mathbb E \Bigg[ \int_0^{t\wedge\tau_x}  \sum_{i<j} 2M_j^{\lambda}(s) \Bigg|K\left(M_i(s),M_j(s)\right)- K \left(\tilde M_i(s),\tilde M_j(s)\right) \Bigg|   ds \Bigg]  \nonumber \\
& \leq & 2 \kappa_{a} \mathbb E \Bigg[ \int_0^{t\wedge\tau_x}  \sum_{i<j} M_j^{\lambda}(s)  \Big( \left|  M_i^{\lambda}(s) - \tilde M_i^{\lambda}(s)\right|\nonumber  + \left|M_j^{\lambda}(s) - \tilde M_j^{\lambda}(s)\right| \Big)ds \Bigg]  \nonumber \\
& \leq & 2\kappa_{a}\mathbb E \left[ \int_0^{t\wedge\tau_x}   \sum_{i\geq 1} \left|  M_i^{\lambda}(s) - \tilde M_i^{\lambda}(s)\right| \sum_{j\geq i+1} M_j^{\lambda}(s)ds\right] \nonumber \\
& &  +2\kappa_{a}\mathbb E \left[  \int_0^{t\wedge\tau_x}   \sum_{j\geq 2} \left|M_j^{\lambda}(s) - \tilde M_j^{\lambda}(s)\right|  \sum_{i=1}^{j-1} M_i^{\lambda}(s) ds \right] \nonumber \\
& \leq & 4 \kappa_{a} \mathbb E \left[\int_0^{t\wedge\tau_x}    \|M(s)\|_{\lambda} \, \delta_{\lambda}\left(M(s),\tilde M(s)\right)ds \right] \nonumber \\
& \leq & 4 \kappa_{a} \,x\,\int_0^{t} \mathbb E \left[ \,\sup_{u\in[0,s\wedge\tau_x]} \delta_{\lambda}\left(M(u),\tilde M(u)\right) \right]ds,
\end{eqnarray}
we used that for $m \in \ell_{\lambda} $, $\sum_{i=1}^{j-1} m_j^{\lambda} \leq \sum_{i=1}^{j-1} m_i^{\lambda} \leq \|m\|_{\lambda}$.\medskip

\noindent Term $C_t^c$: it is treated exactly as $B_t^c$.\medskip

\noindent Term $A_t^f$: We take the expectation, and use (\ref{F3_01}) together with (\ref{Clambdabounds}), to obtain  
\begin{eqnarray}\label{Atf_01}
\mathbb E \left[\sup_{s\in[0,t\wedge\tau_x]} A_s^f\right] & \leq & C^{\lambda}_{\beta}\mathbb E \left[\int_0^{t\wedge\tau_x}  \sum_{i\geq 1} \left(F(M_i(s))\wedge F(\tilde M_i(s))\right)  \left| M_i^{\lambda}(s) -\tilde M_i^{\lambda}(s) \right| \right] ds \nonumber \\
& \leq & \overline{F}_m\,C^{\lambda}_{\beta}\mathbb E \left[\int_0^{t\wedge\tau_x} \sum_{i\geq 1}  \left| M_i^{\lambda}(s) -\tilde M_i^{\lambda}(s) \right| \right] ds \nonumber \\
& \leq &  \overline{F}_m\,C^{\lambda}_{\beta} \int_0^t  \mathbb E \left[ \sup_{u\in[0,s\wedge\tau_x]} \delta_{\lambda}\left(M(u),\tilde M(u)\right) \right] ds.
\end{eqnarray}

\noindent Term $B_t^f$: we take the expectation and use (\ref{Hyp_FragK}) (recall $a:= a_m \vee a_{\tilde m}$), (\ref{inetrian_F}), (\ref{F2_01}) together with (\ref{Clambdabounds}), (\ref{ineq_distance}) and finally Proposition \ref{Result_Prop}. \textit{\ref{Prop_C01_ii})}, to obtain
\begin{eqnarray}\label{Btf_01}
\mathbb E \left[ \sup_{s\in[0,t\wedge\tau_x]}B_s^f\right] & \leq & C^{\lambda}_{\beta}\mathbb E \left[\int_0^{t\wedge\tau_x}  \sum_{i\geq 1} \left| F(M_i(s)) - F(\tilde M_i(s))\right| M_i^{\lambda}(s) \right] ds \nonumber \\
& \leq & \mu_{a}\, C^{\lambda}_{\beta} \mathbb E \left[\int_0^{t\wedge\tau_x}  \sum_{i\geq 1} \left| M_i(s)^{\alpha} - \tilde M_i(s)^{\alpha} \right| \left(M_i^{\lambda}(s) + \tilde M_i^{\lambda}(s)\right) \right] ds \nonumber\\
& \leq &  \mu_{a}\, C^{\lambda}_{\beta}\, C \,  \mathbb E \Bigg[\int_0^{t\wedge\tau_x}  \left(\|M(s)\|_1^{\alpha} + \|\tilde M(s)\|_1^{\alpha} \right) \nonumber\times \sum_{i\geq 1}  \left| M_i^{\lambda}(s) -\tilde M_i^{\lambda}(s) \right|\Bigg] ds \nonumber \\
& \leq & 2\mu_{a}\, C^{\lambda}_{\beta}\, C \left(\|m\|_{1}^{\alpha} \vee \|\tilde m\|_{1}^{\alpha}\right) \times\int_0^t  \mathbb E \left[ \sup_{u\in[0,s\wedge\tau_x]} \delta_{\lambda}\left(M(u),\tilde M(u)\right)\right] ds.
\end{eqnarray}

\noindent Term $C_t^f$: it is treated exactly as $B_t^f$.\medskip

\noindent \textit{Conclusion.-} we take the expectation on (\ref{decomp_dist}) and gather (\ref{Btc_01}), (\ref{Atf_01}) and (\ref{Btf_01}) to obtain
\begin{eqnarray}\label{dist_01a}
\mathbb E \left[ \sup_{s\in[0,t\wedge \tau_x]}\delta_{\lambda}\left( M(s),\tilde M(s)\right)  \right] & \leq &  \delta_{\lambda}\left( m,\tilde m \right) \nonumber \\
&& +  \left[ 8\kappa_{a}\,x  +  4\mu_{a}\, C^{\lambda}_{\beta}\, C \left(\|m\|_{1}^{\alpha} \vee \|\tilde m\|_{1}^{\alpha}\right) + \overline{F}_mC^{\lambda}_{\beta}\right]\nonumber \\
& & \hspace{1cm}\times \int_0^t \mathbb E \left[  \,\sup_{u\in[0,s\wedge\tau_x]} \delta_{\lambda}\left(M(u),\tilde M(u)\right)   \right] ds.
\end{eqnarray}
We conclude using the Gronwall Lemma:
\begin{eqnarray*}
\mathbb E \left[ \sup_{s\in[0,t\wedge \tau_x]}\delta_{\lambda}\left( M(s),\tilde M(s)\right)  \right] & \leq &  \delta_{\lambda}\left( m,\tilde m \right) \times e^{C\,\left(x \vee 1\vee \|m\|_{1}^{\alpha} \vee \|\tilde m\|_{1}^{\alpha}\right)\, t }\nonumber\\
& \leq &  \delta_{\lambda}\left( m,\tilde m \right) \, e^{C\,\left(x+1\right)\, t }.
\end{eqnarray*}
Where $C$ is a positive constant depending on $\lambda$, $\alpha$, $\kappa_a$, $\mu_a$, $K$, $F$, $C^{\lambda}_{\beta}$, $\|m\|_1$ and $\|\tilde m\|_1$.\medskip

This ends the proof of Proposition \ref{Result_Prop}.
\end{proof}

\subsection{Existence for \textit{SDE}: general case}\label{EandUSDE}

We may now prove existence for $(SDE)$. For this, we will build a sequence of coupled finite Coalescence-Fragmentation processes which will be proved to be a Cauchy sequence in $\mathbb D\left([0,\infty),\ell_{\lambda}\right)$. 

\begin{theorem}\label{The_EandU}
Let $\lambda \in(0,1]$, $\alpha\geq 0$ and $m\in\ell_{\lambda}$. Consider the coagulation kernel $K$, the fragmentation kernel $F$, the measure $\beta$ and the Poisson measures $\mathcal N$ and $\mathcal M$ as in Definition \ref{Def_Proc}. \medskip

\noindent Then, there exists a solution $(M(m,t))_{t\geq0}$ to $SDE(K,F,m,\mathcal N,\mathcal M)$.
\end{theorem}

We point out that we do not provide a pathwise uniqueness result for such processes. This is because, under our assumptions, we cannot take advantage of Proposition \ref{Result_Prop}. for this process since the expressions in (\ref{Proof:M1}), (\ref{Proof:Ml}) and (\ref{decomp_dist}) are possibly not true in general.\medskip

Nevertheless, when adding the hypothesis $\lim_{x+y \rightarrow 0}K(x,y) = 0$ to the coagulation kernel we can prove that these expressions hold by considering finite sums and passing to the limit.  We believe that this is due to a possible injection of \textit{dust} (particles of mass $0$) into the system which could produce an increase in the total mass of the system; see \cite{Sum_Kernels}.\medskip 

\noindent In order to prove this theorem, we first need the following lemma.
\begin{lemma}\label{tainfini}
Let $\lambda \in (0,1]$ and $\alpha\geq 0$ be fixed. Assume that the coagulation kernel $K$, the fragmentation kernel $F$ and a measure $\beta$ satisfy Hypotheses \ref{HypSto}. Consider for all $k\geq 1$ the measure $\beta_k$ defined by (\ref{beta_n}). Finally, consider also a subset $\mathcal A$ of $\ell_{0+}$ such that $\sup_{m\in\mathcal A} \|m\|_{\lambda} < \infty$ and $\lim_{i\rightarrow \infty} \sup_{m\in\mathcal A} \sum_{k\geq i} m^{\lambda}_{k} =0$. \medskip

\noindent For each $m\in\mathcal A$ and each $k\geq 1$, let $(M^k(m,t))_{t\geq 0}$ be the unique solution to $SDE(K,F,m,\mathcal N,\mathcal M_k)$ constructed in Lemma \ref{lemmaFinite}., define $\tau_k(m,x) = \inf \{t\geq 0:\|M^k(m,t)\|_{\lambda} \geq x\}$. Then for each $t\geq 0$ we have $\underset{x \rightarrow \infty}{\lim} \gamma(t,x) = 0$, where

\begin{equation}
 \gamma(t,x) := \sup_{m\in \mathcal A} \sup_{k\geq 1} P\left[\sup_{s\in[0,t]}\|M^k(m,s)\|_{\lambda} \geq x\right]. \nonumber
\end{equation}
\end{lemma}
Remark that this convergence does not depend on $\beta_k$ since is based on a bound not depending in the number of fragments but only on $C_{\beta}^{\lambda}$. 

\subsubsection{Proofs}
\begin{proof}[Proof of Lemma \ref{tainfini}]
It suffices to remark that from Proposition \ref{Result_Prop}. \textit{\ref{Prop_C01_i})}, we have
\begin{eqnarray*}
\sup_{m\in \mathcal A} \sup_{k\geq 1} P\left[\sup_{[0,t]}\|M^k(m,s)\|_{\lambda} \geq x\right] & \leq &  \frac{1}{x} \sup_{m\in \mathcal A} \sup_{k\geq 1}\mathbb E\left[ \sup_{[0,t]}\|M^k(m,s)\|_{\lambda}\right] \\
& \leq & \frac{1}{x} \sup_{m\in \mathcal A} \|m\|_{\lambda} e^{\overline F_m C^{\lambda}_{\beta} t}.
\end{eqnarray*}
We make $x$ tend to infinity and the lemma follows.
\end{proof}
\begin{proof}[Proof of Theorem \ref{The_EandU}]

First, recall $\psi_n$ defined by (\ref{projector}) and the measure $\beta_n = \mathds 1 _{\theta \in\Theta(n)} \beta \circ \psi_n^{-1}$. Consider the Poisson measure $\mathcal M(dt,di,d\theta,dz)$ associated to the fragmentation, as in Definition \ref{Def_Proc}.\medskip

\noindent We set $\mathcal M_n = \mathds 1 _{\Theta(n)}\mathcal  M \circ \psi_n^{-1}$. This means that writing $\mathcal M$ as $\mathcal M = \sum_{k \geq 1} \delta_{(T_k,i_k,\theta_k,z_k)}$, we have $\mathcal M_n = \sum_{k \geq 1} \delta_{(T_k,i_k,\psi_n(\theta_k),z_k)} \mathds 1_{\theta\in\Theta(n)}$. Defined in this way, $\mathcal M_n$ is a Poisson measure on $[0,\infty) \times \mathbb N \times \Theta \times [0,\infty)$ with intensity measure  $dt$ $\left(\sum_{k \geq 1} \delta_{k}(di) \right)$ $\beta_n(d\theta)$ $dz$. In this paragraph $\delta_{(\cdot)}$ holds for the Dirac measure on $(\cdot)$.\medskip

We define $m^n \in \ell_{0+}$ by $m^n = (m_1,m_2,\cdots,m_n,0,\cdots)$ and denote $M^n(t):=M(m^n,t)$ the unique solution to $SDE(K,F,m^n,\mathcal N,\mathcal M_n)$ obtained in Proposition \ref{Finite_PoissonProc}. Note that $M^n(t)$ satisfies the following equation
\begin{eqnarray}\label{finiproc01}
\\
M^n(t) & = & m^n + \int_0^t \int_{i<j} \int_0^{\infty} \left[ c_{ij}\left(M^n(s-)\right) - M^n(s-)\right] \mathds 1_{\lbrace z \leq K(M^n_i(s-),M^n_j(s-))\rbrace} \nonumber \\
&&\hspace{10cm} \mathcal N(ds,d(i,j),dz)\nonumber \\
&&+ \int_0^t \int_{i} \int_{\Theta} \int_0^{\infty} \left[ f_{i\psi_n(\theta)}\left(M^n(s-)\right) - M^n(s-)\right] \mathds 1_{ \lbrace z \leq F(M^n_i(s-))\rbrace} \mathds 1_{ \lbrace \theta \in \Theta(n)\rbrace}\nonumber\\
&& \hspace{10cm} \mathcal M(ds,di,d\theta,dz).\nonumber
\end{eqnarray}
This setting allows us to couple the processes since they are driven by the same Poisson measures. \medskip

\noindent \textbf{Convergence  $M^n_t \rightarrow M_t$.--} Consider $p, q \in \mathbb N$ with $1\leq p < q$, from (\ref{finiproc01}) we obtain
\begin{eqnarray}\label{dl_nn+1}
\delta_{\lambda}\left(M^p(t),M^{q}(t)\right) & \leq & \delta_{\lambda}(m^p, m^{q}) + A_c^{p,q}(t) + B_c^{p,q}(t)  + C_c^{p,q}(t) \\
& &  + A_f^{p,q}(t) + B_f^{p,q}(t) + C_f^{p,q}(t) + D_f^{p,q}(t).\nonumber
\end{eqnarray}
We obtain this equality, exactly as in (\ref{decomp_dist}), by replacing $M$ by $M^p$ and $\tilde M$ by $M^{q}$. The terms concerning the coalescence are the same. The terms concerning the fragmentation are, equivalently: 

\begin{eqnarray*}
A_f^{p,q}(t) & = & \int_0^t \int_{i} \int_{\Theta}  \int_0^{\infty} \bigg\lbrace \delta_{\lambda}\left(f_{i\psi_p(\theta)}(M^p(s-)),f_{i\psi_p(\theta)}(M^{q}(s-))\right) \\
&&\hspace{2cm} - \delta_{\lambda}\left(M^p(s-), M^{q}(s-)\right)\bigg\rbrace  \mathds 1_{\{\theta\in\Theta(p)\}} \mathds 1_{\left\lbrace z \leq F\left(M_i^p(s-)\right) \wedge F\left(M_i^{q}(s-)\right) \right\rbrace}\\
& &\hspace{9.5cm}\mathcal  M(ds,di,d\theta,dz),
\end{eqnarray*}
\begin{eqnarray*}
B_f^{p,q}(t) & = & \int_0^t \int_{i} \int_{\Theta} \int_0^{\infty} \left\lbrace \delta_{\lambda}\left(f_{i\psi_p(\theta)}(M^p(s-)), M^{q}(s-)\right) - \delta_{\lambda}\left(M^p(s-), M^{q}(s-)\right)\right\rbrace \\
& & \hspace{2.3cm}\mathds 1_{\{\theta\in\Theta(p)\}} \mathds 1_{\left\lbrace F\left(M_i^{q}(s-)\right)  \leq z \leq F\left(M_i^p(s-)\right) \right\rbrace}\,  \mathcal M(ds,di,d\theta,dz),
\end{eqnarray*}
\begin{eqnarray*}
C_f^{p,q}(t) & = & \int_0^t \int_{i} \int_{\Theta} \int_0^{\infty} \left\lbrace \delta_{\lambda}\left(f_{i\psi_p(\theta)}(M^q(s-)), M^{p}(s-)\right) - \delta_{\lambda}\left(M^p(s-), M^{q}(s-)\right)\right\rbrace \\
& & \hspace{2.3cm}\mathds 1_{\{\theta\in\Theta(p)\}} \mathds 1_{\left\lbrace F\left(M_i^{p}(s-)\right)  \leq z \leq F\left(M_i^q(s-)\right) \right\rbrace}\,  \mathcal M(ds,di,d\theta,dz),
\end{eqnarray*}
Finally, the term $D_f^{p,q}(t)$ is the term that collects the errors. 
\begin{eqnarray*}
D_f^{p,q}(t) & = &  \int_0^t \int_{i} \int_{\Theta}  \int_0^{\infty} \delta_{\lambda}\left(f_{i\psi_p(\theta)}(M^{q}(s-)),f_{i\psi_{q}(\theta)}(M^{q}(s-))\right) \mathds 1_{\{\theta\in\Theta(p)\}}\\
& & \hspace{5cm}\mathds 1_{\left\lbrace z \leq F\left(M_i^q(s-)\right) \right\rbrace}\, \mathcal M(ds,di,d\theta,dz)\\
&& + \int_0^t \int_{i} \int_{\Theta}  \int_0^{\infty} \left\lbrace \delta_{\lambda}\left(f_{i\psi_{q}(\theta)}(M^{q}(s-)), M^p(s-)\right) - \delta_{\lambda}\left(M^{p}(s-), M^{q}(s-)\right)\right\rbrace \\
& & \hspace{3cm} \mathds 1_{\left\lbrace z \leq F\left(M_i^q(s-)\right) \right\rbrace}\,\mathds 1_{\{\theta \in \Theta(q)\setminus \Theta(p)\}}\, \mathcal M(ds,di,d\theta,dz).
\end{eqnarray*}
The first term of $D_f^{p,q}(t)$ results from the utilization of the triangle inequality that gives $A_f^{p,q}(t)$ and $C_f^{p,q}(t)$. The second term is issued from fragmentation of $M^{q}$ when $\theta$ belongs to $\Theta(q) \setminus \Theta(p)$. This induces a fictitious jump to $M^p$ which does not undergo fragmentation.\medskip

\noindent We proceed to bound each term. We define, for all $x >0$ and $n\geq1$, the stopping time $\tau^x_n = \inf \{t\geq 0:\|M^n(t)\|_{\lambda} \geq x\}$.\medskip

From Proposition \ref{Result_Prop}. we have for all $s \in[0,t]$, $$\sup_{n\geq 1} \sup_{i\geq 1} M_i^n(s) \leq \sup_{n\geq 1} \sup_{i\geq 1}\sup_{[0,t]} \|M^n(s)\|_1 \leq \|m\|_1 := a_m\, a.s.$$

We set $\kappa_{a_m}$ and $\mu_{a_m}$ the constants for which the kernels $K$ and $F$ satisfy (\ref{Hyp_CoagK}) and (\ref{Hyp_FragK}). Finally, we set $\overline F_m = \sup_{[0,a_m]}F(x)$.  \medskip

\noindent The terms concerning coalescence are upper bounded on $[0,t \wedge \tau^x_{p} \wedge\tau^x_q]$ with $t\geq 0$, exactly as in (\ref{decomp_dist}).\medskip

\noindent Term  $A_f^{p,q}(t)$: we take the $\sup$ on $[0,t \wedge \tau^x_{p} \wedge\tau^x_q]$ and then the expectation. We use (\ref{F3_01}) together with (\ref{Clambdabounds}). We thus obtain exactly the same bound as for $A^f_t$. \medskip

\noindent Term $B_f^{p,q}(t)$: we take the $\sup$ on $[0,t \wedge \tau^x_{p} \wedge\tau^x_q]$ and then the expectation. We use (\ref{inetrian_F}), (\ref{F2_01}) with (\ref{Clambdabounds}) and (\ref{Hyp_FragK}). We thus obtain exactly the same bound as for $B^f_t$.\medskip

\noindent Term $C_f^{p,q}(t)$: it is treated exactly as $B_f^{p,q}(t)$. \medskip

\noindent Term $D_f^{p,q}(t)$: we take the $\sup$ on $[0,t \wedge \tau^x_{p} \wedge\tau^x_q]$ and then the expectation. For the first term we use (\ref{dln}). For the second term we use (\ref{inetrian_F}) and (\ref{F2_01}) together with (\ref{Clambdabounds}). Finally, we use Proposition \ref{Result_Prop}. \textit{\ref{Prop_C01_i})}. and the notation $C(\theta) := \sum_{k\geq 2} \theta_k^{\lambda} + (1-\theta_1^{\lambda})$, to obtain

\begin{eqnarray*}
&&\mathbb E \left[\sup_{s\in[0,t\wedge \tau^x_{q} \wedge\tau^x_p]} D_f^{p,q}(t)\right] \\
& &\hspace{0.5cm} \leq\,\,  \mathbb E \left[\int_0^{t \wedge \tau^x_{p} \wedge\tau^x_{q}} \sum_{i\geq 1} F\left(M^q_i(s)\right) \int_{\Theta}\mathds 1_{\{\theta \in \Theta(p)\}}\sum_{k = p+1}^q\theta_{k}^{\lambda}[M_i^{q}(s)]^{\lambda}\beta(d\theta) ds\right] \nonumber \\
&&\hspace{1cm} \,\,+\,\,  \mathbb E \Bigg[\int_0^{t\wedge \tau^x_{p} \wedge\tau^x_q}  \sum_{i\geq 1}  F\left(M^q_i(s)\right) [M_i^{q}(s)]^{\lambda} ds\, \int_{\Theta} C(\theta) \mathds 1_{\{\theta \in \Theta(q)\setminus \Theta(p)\}}\beta(d\theta)\Bigg]\nonumber \\
&& \hspace{0.5cm} \leq\,\, \overline{F}_m \int_{\Theta} \sum_{k > p}\theta_{k}^{\lambda}\beta(d\theta)\,\int_0^t\mathbb E \left[ \sup_{u\in[0,t]} \|M^{q}(u)\|_{\lambda} \right]  ds\nonumber\\
&& \hspace{1cm} \,\,+\,\,\overline{F}_m\, \int_{\Theta}  C(\theta) \mathds 1_{\{\theta \in \Theta\setminus \Theta(p)\}}\beta(d\theta)\,\int_0^t\mathbb E \left[ \sup_{u\in[0,t]} \|M^{q}(u)\|_{\lambda} \right]  ds\nonumber\\
& & \hspace{0.5cm} \leq\,\, \overline{F}_m\,t\, \|m\|_{\lambda}\, e^{\overline{F}_m\,C^{\lambda}_{\beta}\, t}(A(p) + B(p)),\nonumber
\end{eqnarray*}
where $A(p) : =\int_{\Theta} \sum_{k > p}\theta_{k}^{\lambda}\beta(d\theta)$ and $B(p) := \int_{\Theta}C(\theta) \mathds 1_{\{\theta \in \Theta\setminus \Theta(p)\}}\beta(d\theta)$. Note that by (\ref{Hyp2_Beta}) and since $\Theta\setminus \Theta(p)$ tends to the empty set, $A(p)$ and $B(p)$ tend to $0$ as $p$ tends to infinity.\medskip

\noindent Thus, gathering the terms  as for the bound (\ref{dist_01a}), we get
\begin{eqnarray}\label{dist_01na}
&&\mathbb E \left[ \sup_{s\in[0,t \wedge \tau^x_q\wedge  \tau^x_{p}]}\delta_{\lambda}\left( M^p(s), M^{q}(s)\right)  \right] \nonumber \\
&&\hspace{2cm} \leq \,\,  \delta_{\lambda}\left( m^p, m^{q} \right) + D_1 t [A(p) + B(p)] \nonumber\\
&& \hspace{2.5cm}+  \left( 8\kappa_1 \,x  +  C C_{\beta}^{\lambda} \|m\|_{1}^{\alpha} \right)\int_0^t \mathbb E \left[  \,\sup_{u\in[0,s \wedge \tau^x_q\wedge  \tau^x_{p}]} \delta_{\lambda}\left(M^{p}(u), M^{q}(u)\right)   \right] ds,
\end{eqnarray}
where $D_1 = \overline{F}_m \, \|m\|_{\lambda}\, e^{\overline{F}_m \,C^{\lambda}_{\beta}\, t}$. The Gronwall Lemma allows us to obtain
\begin{equation}\label{dist_01nb}
\mathbb E \left[ \sup_{s\in[0,t \wedge \tau^x_q\wedge  \tau^x_{p}]}\delta_{\lambda}\left( M^p(s), M^{q}(s)\right)  \right]  \leq  \left\lbrace \delta_{\lambda}\left( m^p, m^{q} \right) + D_1 [A(p) + B(p)]t \right\rbrace\times e^{D_2 \,x\, t },
\end{equation}
where $D_2$ is a positive constants depending on $\lambda$, $\alpha$, $\kappa_{a_m}$, $\mu_{a_m}$, $K$, $F$, $C^{\lambda}_{\beta}$ and $\|m\|_1$.\medskip

Since $\lim_{n\rightarrow \infty}\delta_{\lambda}(m^n,m) = 0$, we deduce from Lemma \ref{tainfini}. that for all $t\geq 0$,
\begin{equation}\label{lemma_adapted}
\lim_{x\rightarrow \infty}\overline{\gamma}(t,x) = 0 \,\,\, \textrm{where} \,\,\, \overline{\gamma}(t,x) := \sup_{n\geq 1} P[\tau(m^n,x)\leq t].
\end{equation}

This means that the stopping times $\tau^x_n$ tend to infinity as $x\rightarrow \infty$, uniformly in $n$.\medskip

Next, from (\ref{dist_01nb}), (\ref{lemma_adapted}) and since $(m^n)_{n\geq 1}$ is a Cauchy sequence for $\delta_{\lambda}$ and $(A(n))_{n\geq1}$ and $(B(n))_{n\geq1}$ converge to $0$, we deduce that for all $\varepsilon >0$, $T>0$  we may find $n_{\varepsilon}>0$ such that for $p,q\geq n_{\varepsilon}$ we have
\begin{equation}\label{boundProba}
P\left[\sup_{[0,T]}\delta_{\lambda}\left(M^p(t),M^{q}(t)\right) \geq \varepsilon \right] \leq \varepsilon.
\end{equation}
Indeed, for all $x>0$,
\begin{eqnarray*}
P\left[\sup_{[0,T]}\delta_{\lambda}\left(M^p(t),M^{q}(t)\right) \geq \varepsilon \right] &\leq & P[\tau^x_p \leq T] + P[\tau^x_q \leq T] +\frac{1}{\varepsilon} \mathbb E\left[ \sup_{[0,T\wedge \tau^x_p \wedge \tau^x_q]}\delta_{\lambda}\left(M^p(t),M^{q}(t)\right)\right]\\
& \leq &  2\overline{\gamma}(T,x) + \frac{1}{\varepsilon}  [\delta_{\lambda}\left( m^p, m^{q} \right) + D_1 T(A(p)+B(p)) ]\times e^{D_2 \,x\, T }. \nonumber
\end{eqnarray*}
Choosing $x$ large enough so that $\overline{\gamma}(T,x) \leq \varepsilon/8$ and $n_{\varepsilon}$ large enough to have both $A(p)$ and $B(p) \leq (\varepsilon^2/4 D_1\,T)e^{-D_2xT}$ and in a such a way that for all $p,q \geq n_{\varepsilon}$, $\delta_{\lambda}\left( m^p, m^{q} \right) \leq (\varepsilon^2/4)e^{-D_2xT}$, we conclude that (\ref{boundProba}) holds.\medskip

We deduce from (\ref{boundProba}) that the sequence of processes $\left( M^n _t\right)_{t\geq 0}$ is Cauchy in probability in $\mathbb D([0,\infty),\ell_{\lambda})$, endowed with the uniform norm in time on compact intervals. We are thus able to find a subsequence (not relabelled) and a $(\mathcal H_t)$-adapted process $(M(t))_{t\geq 0}$ belonging $a.s.$ to  $\mathbb D([0,\infty),\ell_{\lambda})$ such that for all $T>0$,
\begin{equation}\label{limProbn01}
\lim_{n\rightarrow \infty} \sup_{[0,T]}\delta_{\lambda}\left(M^n(t),M(t)\right) = 0.\,\,\,a.s.
\end{equation}

\noindent Setting now $\tau^x:= \inf\{t\geq 0: \|M(t)\|_{\lambda}\geq x\}$, due to Lebesgue Theorem, 
\begin{equation}\label{limE01}
\lim_{n\rightarrow \infty} \mathbb E\left[\sup_{[0,T\wedge \tau^x_n \wedge \tau^x]}\delta_{\lambda}\left(M^n(t),M(t)\right)\right] = 0.
\end{equation}
We have to show now that the limit process $(M(t))_{t\geq 0}$ defined by (\ref{limProbn01}) solves the equation $SDE(K,F,m,\mathcal N,\mathcal M)$ defined in (\ref{def_sto-process}). \medskip

We want to pass to the limit in (\ref{finiproc01}), it suffices to show that $\lim_{n\rightarrow \infty} \Delta_n(t) = 0$, where
\begin{eqnarray*}
 \Delta_n(t) & = & \mathbb E \Bigg[ \int_0^{t\wedge\tau^x_n\wedge \tau^x} \int_{i<j} \int_0^{\infty} \sum_{k\geq 1} 2^{-k} \big|\left( [c_{ij}\left(M(s-)\right)]_k - M_k(s-)]\right) \mathds 1_{\lbrace z \leq K(M_i(s-),M_j(s-))\rbrace} \\ 
& &  - \left([c_{ij}\left(M^n(s-)\right)]_k - M_k^n(s-)\right)  \mathds 1_{\lbrace z \leq K(M_i^n(s-),M_j^n(s-))\rbrace} \big|\mathcal  N(ds,d(i,j),dz) \\
&&+ \int_0^{t\wedge\tau^x_n\wedge \tau^x} \int_{i} \int_{\Theta} \int_0^{\infty}\sum_{k\geq 1} 2^{-k} \big| \left( [f_{i\theta}\left(M(s-)\right)]_k - [M(s-)]_k\right) \mathds 1_{ \lbrace z \leq F(M_i(s-))\rbrace} \\
& & - \left( [f_{i\psi_n(\theta)}\left(M^n(s-)\right)]_k - M_k^n(s-)\right) \mathds 1_{ \lbrace z \leq F(M_i^n(s-))\rbrace} \mathds 1_{\{\theta\in\Theta(n)\}}\big| \mathcal M(ds,di,d\theta,dz) \Bigg].
\end{eqnarray*}
Indeed, due to (\ref{limProbn01}), for all $x>0$ and for $n$ large enough, $a.s.$ $\tau^x_n \geq \tau^{x/2}$. Thus $M$ will solve $SDE(K,F,M(0),\mathcal N,\mathcal M)$ on the time interval $[0,\tau^{x/2})$ for all $x>0$, and thus on $[0,\infty)$ since $a.s.$ $\lim_{x\rightarrow\infty} \tau^x = \infty$, because $M\in\mathbb D([0,\infty),\ell_{\lambda})$.\medskip

Note that
\begin{eqnarray*}
&& \big|\left( [c_{ij}\left(M(s)\right)]_k - M_k(s)]\right) \mathds 1_{\lbrace z \leq K(M_i(s),M_j(s))\rbrace} -\left([c_{ij}\left(M^n(s)\right)]_k - M_k^n(s)\right)  \mathds 1_{\lbrace z \leq K(M_i^n(s),M_j^n(s))\rbrace}\big|\\
&&\hspace{1cm}\leq \Big|\left( [c_{ij}\left(M(s)\right)]_k - M_k(s)]\right) - \left([c_{ij}\left(M^n(s)\right)]_k - M_k^n(s)\right)\Big|\mathds 1_{\lbrace z \leq K(M_i(s),M_j(s))\rbrace}\\
&&\hspace{1.5cm} +  \big|[c_{ij}\left(M^n(s)\right)]_k - M_k^n(s) \big| \left|  \mathds 1_{\lbrace z \leq K(M_i(s),M_j(s))\rbrace} - \mathds 1_{\lbrace z \leq K(M_i^n(s),M_j^n(s))\rbrace} \right|
\end{eqnarray*}
and
\begin{eqnarray*}
&&\big| \left( [f_{i\theta}\left(M(s)\right)]_k - M_k(s)\right) \mathds 1_{ \lbrace z \leq F(M_i(s))\rbrace} - \left( [f_{i\psi_n(\theta)}\left(M^n(s)\right)]_k - M_k^n(s)\right) \mathds 1_{ \lbrace z \leq F(M_i^n(s))\rbrace} \mathds 1_{\{\theta\in\Theta(n)\}}\big|\\
&&\hspace{0.7cm}\leq \big| \left( [f_{i\theta}\left(M(s)\right)]_k - M_k(s)\right) - \left( [f_{i\theta}\left(M^n(s)\right)]_k - M_k^n(s)\right) \big|\mathds 1_{ \lbrace z \leq F(M_i(s))\rbrace}\\
&& \hspace{1.2cm}+ \big|\left( [f_{i\theta}\left(M^n(s)\right)]_k - [f_{i\psi_n(\theta)}\left(M^n(s)\right)]_k \right)\big| \mathds 1_{ \lbrace z \leq F(M_i(s))\rbrace}\\
&& \hspace{1.2cm}+\big| [f_{i\psi_n(\theta)}\left(M^n(s)\right)]_k - M_k^n(s) \big| \left|\mathds 1_{ \lbrace z \leq F(M_i(s))\rbrace} - \mathds 1_{ \lbrace z \leq F(M_i^n(s))\rbrace}\right|\\
&& \hspace{1.2cm}+ \big| [f_{i\psi_n(\theta)}\left(M^n(s)\right)]_k - M_k^n(s) \big|\mathds 1_{ \lbrace z \leq F(M_i^n(s))\rbrace} \mathds 1_{\{\theta\in\Theta(n)^c\}},
\end{eqnarray*}
where $\Theta(n)^c = \Theta\setminus \Theta(n)$. We thus obtain the following bound
\begin{equation*}
 \Delta_n(t) \leq A_n^c(t) + B_n^c(t) +A_n^f(t) + B_n^f(t) + C^f_n(t) + D^f_n(t).
\end{equation*}
First, $A_n^c(t) = \sum_{i<j}A^{ij}_n(t)$ with
\begin{eqnarray*}
A_n^{ij}(t) & = & \mathbb E\Bigg[\int_0^{t\wedge\tau^x_n\wedge \tau^x} K\left(M_i(s),M_j(s)\right) \sum_{k\geq 1} 2^{-k}  \\
& & \hspace{1cm} \left|\left( [c_{ij}\left(M(s)\right)]_k - M_k(s)]\right)- \left([c_{ij}\left(M^n(s)\right)]_k - M_k^n(s)\right)\right|ds  \Bigg],
\end{eqnarray*}
and using 
\begin{eqnarray*} 
&&\left|  \mathds 1_{\lbrace z \leq K(M_i(s),M_j(s))\rbrace} - \mathds 1_{\lbrace z \leq K(M_i^n(s),M_j^n(s))\rbrace} \right|  \\
&&  \hspace{1.5cm}=\mathds 1_{\lbrace K(M_i(s),M_j(s))\wedge K(M_i^n(s),M_j^n(s))\leq z \leq K(M_i(s),M_j(s))\vee K(M_i^n(s),M_j^n(s))\rbrace},
\end{eqnarray*}
\begin{eqnarray*}
B_n^c(t) & = & \mathbb E\Bigg[ \int_0^{t\wedge\tau^x_n\wedge \tau^x} \sum_{i<j}\left|K\left(M_i(s),M_j(s)\right)-K\left(M_i^n(s),M_j^n(s)\right) \right| \\
& & \sum_{k\geq 1} 2^{-k}  \left|[c_{ij}\left(M^n(s)\right)]_k - M_k^n(s)\right|ds \Bigg].
\end{eqnarray*}
For the fragmentation terms we have 
\begin{eqnarray*}
A_n^f(t) & = & \mathbb E\Bigg[\int_0^{t\wedge\tau^x_n\wedge \tau^x} \int_{\Theta} \sum_{i\geq 1} F\left(M_i(s)\right) \\
& & \sum_{k\geq 1} 2^{-k}  \left|\left([f_{i\theta}\left(M(s)\right)]_k - M_k(s)\right) - \left([f_{i\theta}\left(M^n(s)\right)]_k - M_k^n(s)\right)\right| \beta(d\theta) ds\Bigg], 
\end{eqnarray*}
\begin{eqnarray*}
B_n^f(t) & = & \mathbb E\Bigg[\int_0^{t\wedge\tau^x_n\wedge \tau^x} \int_{\Theta} \sum_{i\geq 1} F\left(M_i(s)\right)  \sum_{k\geq 1} 2^{-k}  \left|\left([f_{i\theta}\left(M^n(s)\right)]_k - [f_{i\psi_n(\theta)}\left(M^n(s)\right)]_k \right)\right| \beta(d\theta) ds\Bigg], 
\end{eqnarray*}
using
\begin{equation} 
\left|\mathds 1_{ \lbrace z \leq F(M_i(s))\rbrace} - \mathds 1_{ \lbrace z \leq F(M_i^n(s))\rbrace}\right|  = \mathds 1_{\lbrace F(M_i(s))\wedge F(M_i^n(s))\leq z \leq F(M_i(s))\vee F(M_i^n(s))\rbrace}, \nonumber
\end{equation}
\begin{eqnarray*}
C_n^f(t) & = & \mathbb E\Bigg[\int_0^{t\wedge\tau^x_n\wedge \tau^x} \int_{\Theta} \mathds 1_{\{\theta\in\Theta(n)\}}\sum_{i\geq 1}\left|F\left(M_i(s)\right)-F\left(M_i^n(s)\right) \right| \\
& & \sum_{k\geq 1} 2^{-k}  \left|[f_{i\psi_n(\theta)}\left(M^n(s)\right)]_k - M_k^n(s)\right| \beta(d\theta) ds\Bigg],
\end{eqnarray*}
and finally,
\begin{eqnarray*}
D_n^f(t) & = & \mathbb E\Bigg[\int_0^{t\wedge\tau^x_n\wedge \tau^x}\, \int_{\Theta}\,\mathds 1_{\{\theta\in\Theta(n)^c\}} \sum_{i\geq 1} F\left(M_i^n(s)\right) \\
& & \sum_{k\geq 1} 2^{-k}  \left|[f_{i\psi_n(\theta)}\left(M^n(s)\right)]_k - M_k^n(s)\right| \beta(d\theta)ds\Bigg], 
\end{eqnarray*}

We will show that each term converges to $0$ as $n$ tends to infinity. \medskip

Note first that from (\ref{limProbn01}) we have, $a.s.$ $\sup_{[0,t]}\|M(s)\|_1 \leq \underset{n\rightarrow \infty}{\limsup}\sup_{[0,t]}\|M^ n(s)\|_1$ and $a.s.$ $\sup_{[0,t]}\|M(s)\|_{\lambda} \leq\underset{n\rightarrow \infty}{\limsup} \sup_{[0,t]}\|M^ n(s)\|_{\lambda}$ and from Proposition \ref{Result_Prop} \textit{\ref{Prop_C01_i})}, we get $\sup_{n\geq1} \sup_{[0,t]} \|M^n(s)\|_1 \leq \|m\|_1$, implying for all $t\geq0$
\begin{equation}\label{M1limBounded}
\sup_{s\in[0,t]}\|M(s)\|_1 \leq \|m\|_1 :=a_m <\infty, \,\,a.s.,
\end{equation}
equivalently for $M^ n$, we have $a_{m^n} = \|m^n\|_1 \leq  \|m\|_1$. We set $\kappa_{a_m}$ and $\mu_{a_m}$ the constants for which the kernels $K$ and $F$ satisfy (\ref{Hyp_CoagK}) and (\ref{Hyp_FragK}). Finally, we set $\overline K_m = \sup_{[0,a_m]^2}K(x,y)$ and $\overline F_m = \sup_{[0,a_m]}F(x)$.  \medskip

We prove that $A_n^c(t)$ tends to $0$ using the Lebesgue dominated convergence Theorem. It suffices to show that:
\begin{enumerate}[a)]
\item for each $1\leq i < j$, $A_n^{ij}(t)$ tends to $0$ as $n$ tends to infinity,
\item \label{2c} $\lim_{k\rightarrow \infty} \limsup_{n\rightarrow \infty} \sum_{i+j \geq k}A_n^{ij}(t) = 0$.
\end{enumerate}
Now, for $A_n^{ij}(t)$ using (\ref{d2c}), (\ref{d_dlambda}), (\ref{M1limBounded}) and Proposition \ref{Result_Prop}. \textit{\ref{Prop_C01_i})}, we have
\begin{eqnarray*}
A_n^{ij}(t) & \leq & \overline{K}_m\mathbb E\left[ \int_0^{t\wedge\tau^x_n\wedge \tau^x} d\left(c_{ij}\left(M(s)\right),c_{ij}\left(M^n(s)\right)\right) + d\left(M(s),M^n(s)\right)\,ds  \right] \\
&\leq & \overline{K}_m \mathbb E\left[\int_0^{t\wedge\tau^x_n\wedge \tau^x} \left(2^{i} + 2^{j} + 1\right)  d\left(M(s),M^n(s)\right)\,ds  \right] \\
&\leq & C \,\overline{K}_m \left(2^{i} + 2^{j} + 1\right) \mathbb E\Bigg[   \int_0^{t\wedge\tau^x_n\wedge \tau^x} \left(\|M(s)\|_1^{1-\lambda} \vee \|M^n(s)\|_1^{1-\lambda}\right) \\
&& \hspace{7.5cm} \times \delta_{\lambda}\left(M(s),M^n(s)\right)\,ds \Bigg] \\
&\leq & C \,\overline{K}_m \left(2^{i} + 2^{j} + 1\right) t \|m\|_1^{1-\lambda}   \mathbb E\left[ \sup_{[0,t\wedge\tau^x_n\wedge \tau^x]} \delta_{\lambda}\left(M(s),M^n(s)\right) \right].
\end{eqnarray*}
which tends to $0$ as $n\rightarrow \infty$ due to (\ref{limE01}). On the other hand, using (\ref{d1c}) we have
\begin{eqnarray*}
A_n^{ij}(t) & \leq & \overline{K}_m  \mathbb E\left[ \int_0^{t\wedge\tau^x_n\wedge \tau^x} d\left(c_{ij}\left(M(s)\right),M(s)\right) + d\left(c_{ij}\left(M^n(s)\right),M^n(s)\right) \,ds\right] \\
&\leq &\frac{3\overline{K}_m}{2} 2^{-i} \int_0^t\, \mathbb E\left[  M_j(s)+M_j^n(s)\right]ds. 
\end{eqnarray*}
Since $\sum_{i\geq 1}2^{-i} = 1$ and $\sum_{j\geq1} \int_0^t\mathbb E[M_j(s)]ds\leq \|m\|_1 t$, \ref{2c}) reduces to
\begin{equation}
\lim_{k\rightarrow \infty} \limsup_{n\rightarrow \infty} \sum_{j\geq k} \int_0^t  \mathbb E[M_j^n(s)] ds = 0.\nonumber
\end{equation}
For each $k\geq 1$, since $M^n(s)$ and $M(s)$ belong to $\ell_1$ for all $s\geq 0$ $a.s$ and since the map $m \mapsto \sum_{j=1}^{k-1}m_j$ is continuous for the pointwise convergence topology,

\begin{eqnarray*}
\limsup_{n\rightarrow \infty} \int_0^t \mathbb E\left[\sum_{j\geq k} M^n_j(s) \right] & = & \int_0^t ds \left\lbrace \lim_{n\rightarrow \infty}\|M^n(s)\|_1 -\lim_{n\rightarrow \infty} \mathbb E\left[\sum_{j=1}^{k-1} M^n_j(s) \right]\right \rbrace ds\\
& = & \int_0^t \left\lbrace \|M(s)\|_1 - \mathbb E\left[\sum_{j=1}^{k-1} M_j(s) \right]\right \rbrace ds\\
& = & \int_0^t\mathbb E\left[\sum_{j=k}^{\infty} M_j(s) \right] ds.
\end{eqnarray*}
We easily conclude using that $a.s.$ $\|M(s)\|_1 <\|m\|_1$ for all $s\geq 0$.\medskip

Using (\ref{Hyp_CoagK}), (\ref{d1c}) and  Proposition \ref{Result_Prop}. \textit{\ref{Prop_C01_i})}, we obtain
\begin{eqnarray*}
B_n^c(t) & \leq & \kappa_{a_m}  \mathbb E\Bigg[ \int_0^{t\wedge\tau^x_n\wedge \tau^x}\sum_{i<j}\left[\left|M_i^n(s) ^{\lambda}- M_i(s)^{\lambda}\right|  + \left|M_j^n(s)^{\lambda} - M_j(s)^{\lambda}\right|\,ds \right] \\
& &\hspace{8cm}\times d\left(c_{ij}\left(M^n(s)\right), M^n(s)\right) \Bigg]\\
&\leq &  \frac{3}{2}\kappa_{a_m} \mathbb E\Bigg[\int_0^{t\wedge\tau^x_n\wedge \tau^x} \sum_{i<j}\Big[\left|M_i^n(s) ^{\lambda}- M_i(s)^{\lambda}\right|+ \left|M_j^n(s)^{\lambda} - M_j(s)^{\lambda}\right|\Big] 2^{-i} M_j^n(s)\,ds \Bigg] \\
& \leq & 3t\,\kappa_{a_m} \|m\|_1 \mathbb E\left[\sup_{[0,[t\wedge\tau^x_n\wedge \tau^x]}\delta_{\lambda}\left(M(s),M^n(s)\right) \right],
\end{eqnarray*}
which tends to $0$ as $n\rightarrow \infty$ due to (\ref{limE01}).\medskip

We use (\ref{d2f}) and (\ref{d_dlambda}) both with (\ref{M1limBounded}) and Proposition \ref{Result_Prop}. \textit{\ref{Prop_C01_i})} and (\ref{d1f}) to obtain
\begin{eqnarray*}
A_n^f(t) & \leq &  \overline{F}_m \mathbb E \Bigg[ \int_0^{t\wedge\tau^x_n\wedge \tau^x}  \sum_{i\geq 1} \int_{\Theta}\bigg[\bigg(d\left(f_{i\theta}\left(M(s)\right), f_{i\theta}\left(M^n(s)\right)\right) + d\left(M(s), M^n(s)\right)\bigg) \\
&& \hspace{2cm}\wedge\bigg(d\left(f_{i\theta}\left(M(s)\right), M(s)\right) + d\left(f_{i\theta}\left(M^n(s)\right), M^n(s)\right)\bigg) \bigg] \beta(d\theta)ds\Bigg] \\
&\leq & \overline{F}_m \mathbb E \Bigg\lbrace \int_0^{t\wedge\tau^x_n\wedge \tau^x}  \int_{\Theta} \sum_{i\geq 1}\Bigg[\bigg( 2C\|m\|_{\lambda}^{1-\lambda}\, \delta_{\lambda}\left(M(s), M^n(s)\right)\bigg) \wedge \\
&& \hspace{5.3cm}\bigg(2^{-i}(1-\theta_1)\left(M_i(s) + M_i^n(s)\right) \bigg) \Bigg] \beta(d\theta)ds\Bigg\rbrace. 
\end{eqnarray*}
We split the integral on $\Theta$ and the sum on $i$ into two parts. Consider $\Theta_{\varepsilon} = \{ \theta\in\Theta : \theta_1 \leq 1 -\varepsilon\}$ and $N\in\mathbb N$. Using (\ref{M1limBounded}) and Proposition \ref{Result_Prop}. \textit{\ref{Prop_C01_i})} and relabelling the constant $C$, we deduce
\begin{eqnarray*}
&& \int_{\Theta} \sum_{i\geq 1}\bigg[\bigg(C\|m\|_{\lambda}^{1-\lambda}\, \delta_{\lambda}\left(M(s), M^n(s)\right)\bigg) \wedge \bigg(2^{-i}(1-\theta_1)\left(M_i(s) + M_i^n(s)\right) \bigg) \bigg]\beta(d\theta)\\
&&  \hspace{1cm}\leq C\|m\|_{\lambda}^{1-\lambda} \int_{\Theta_{\varepsilon}} \sum_{i=1}^N \delta_{\lambda}\left(M(s), M^n(s)\right) \beta(d\theta) + \int_{\Theta_{\varepsilon}^c}(1-\theta_1)\beta(d\theta) \sum_{i\geq 1} \left(M_i(s)+ M_i^n(s)\right) \\
&& \hspace{1.5cm} + \int_{\Theta}\sum_{i> N}2^{-i} (1-\theta_1)\left(M_i(s)+ M_i^n(s)\right) \beta(d\theta) \\
&& \hspace{1cm} \leq C\|m\|^{1-\lambda}_1 N \beta(\Theta_{\varepsilon}) \delta_{\lambda}\left(M(s), M^n(s)\right) + 2 \|m\|_1  \int_{\Theta_{\varepsilon}^c}(1-\theta_1) \beta(d\theta)   \\
&& \hspace{1.5cm} + \,\,2 \|m\|_1  \int_{\Theta}(1-\theta_1)\beta(d\theta) \sum_{i> N} 2^{-i}.
\end{eqnarray*}
Note that $ \beta(\Theta_{\varepsilon}) = \int_{\Theta} \mathds 1_{\{1-\theta_1 \geq \varepsilon\}} \, \beta(d\theta) \leq \frac{1}{\varepsilon} \int_{\Theta} (1-\theta_1) \, \beta(d\theta) \leq\frac{1}{\varepsilon} \, C^{\lambda}_{\beta} < \infty$. Thus, we get 
\begin{eqnarray*}
A_n^f(t) & \leq &  \frac{t}{\varepsilon} C^{\lambda}_{\beta} N \overline{F}_m  C\|m\|^{1-\lambda}_1  \mathbb E\left[\sup_{[0,[t\wedge\tau^x_n\wedge \tau^x]}\delta_{\lambda}\left(M(s),M^n(s)\right) \right]  \\
&& + 2 t \overline{F}_m \|m\|_1  \int_{\Theta_{\varepsilon}^c}(1-\theta_1) \beta(d\theta) + 4 t \overline{F}_m \|m\|_1  C^{\lambda}_{\beta}\, 2^{-N}.
\end{eqnarray*}
Thus, due to (\ref{limE01}) we have for all $\varepsilon >0$ and $N \geq 1$,

\begin{equation*}
\limsup_{n \rightarrow \infty} A_n^f(t) \leq 2 t \overline{F}_m \|m\|_1  \int_{\Theta_{\varepsilon}^c}(1-\theta_1) \beta(d\theta) + 4 t \overline{F}_m \|m\|_1  C^{\lambda}_{\beta}\, 2^{-N}.
\end{equation*}
Since $\Theta_{\varepsilon}^c$ tends to the empty set as $\varepsilon \rightarrow 0$ we conclude using (\ref{Clambdabounds}) with (\ref{Hyp2_Beta}) and making $\varepsilon \rightarrow 0$ and $N \rightarrow \infty$.\medskip

Next, use (\ref{d3f}) and  Proposition \ref{Result_Prop}. \textit{\ref{Prop_C01_i})} to obtain  
\begin{equation*}
B_n^f(t) \leq  t \overline{F}_t \|m\|_1 \int_{\Theta}\,\sum_{k>n}\theta_k \beta(d\theta).
\end{equation*}
which tends to $0$ as $n\rightarrow \infty$ due to (\ref{Hyp1_Beta}).\medskip

\noindent Using (\ref{Hyp_FragK}), (\ref{d1f}) with (\ref{inebeta}) and (\ref{Clambdabounds}), (\ref{ineq_distance}), (\ref{d_dlambda}), (\ref{M1limBounded}) and  Proposition \ref{Result_Prop}. \textit{\ref{Prop_C01_i})}, we obtain
\begin{eqnarray*}
C_n^f(t) & \leq & 2\mu_{a_m} \mathbb E \left[ \int_0^{t\wedge\tau^x_n\wedge \tau^x}\int_{\Theta(n)}\,\sum_{i\geq 1}\left|[M_i(s)]^{\alpha}-[M_i^n(s)]^{\alpha} \right| 2^{-i} (1-\theta_1) M_i(s)\beta(d\theta)ds  \right]\\
&\leq & 2 \mu_{a_m}\,C^{\lambda}_{\beta} \mathbb E \left[ \int_0^{t\wedge\tau^x_n\wedge \tau^x} \sum_{i\geq 1}2^{-i}\left|M_i(s)-M_i^n(s)\right| \left([M_i^n(s)]^{\alpha} + [M_i(s)]^{\alpha} \right)\,ds\right]\\
&\leq & 2\mu_{a_m}\,C\, C^{\lambda}_{\beta} t \|m\|_1^{1-\lambda+\alpha}  \mathbb E \left[\sup_{[0,t\wedge\tau^x_n\wedge \tau^x]} \delta_{\lambda}\left(M(s),M^n(s)\right) \right],
\end{eqnarray*}
which tends to $0$ as $n\rightarrow \infty$ due to (\ref{limE01}).\medskip

\noindent Finally, we use (\ref{d1f}) with (\ref{inebeta}) and (\ref{Clambdabounds}) and  Proposition \ref{Result_Prop}. \textit{\ref{Prop_C01_i})}, to obtain
\begin{equation}
D_n^f(t) \leq 2t\,\overline{F}_t \|m\|_1  \int_{\Theta}\mathds 1_{\{\theta\in\Theta(n)^c\}} (1-\theta_1)\beta(d\theta), \nonumber
\end{equation}
which tends to $0$ as $n$ tends to infinity since $\int_{\Theta}\left(1-\theta_1\right)\beta(d\theta)\leq C^{\lambda}_{\beta}$ and $\Theta(n)^c$ tends to the empty set.\medskip

This ends the proof of Theorem \ref{The_EandU}.
\end{proof}

\subsection{Conclusion}

It remains to conclude the proof of Theorem \ref{Result_Theo}.\medskip

\noindent 
We start with some boundedness of the operator $\mathcal L_{K,F}^{\beta}$.

\begin{lemma}\label{LemmaLBounded}
Let $\lambda \in(0,1]$, $\alpha\geq 0$, the coagulation kernel $K$, fragmentation kernel $F$ and the measure $\beta$ satisfying Hypotheses \ref{HypSto}. Let $\Phi:\ell_{\lambda}\rightarrow \mathbb R$ satisfy, for all $m, \tilde{m} \in\ell_{\lambda}$, $|\Phi(m)|\leq a$ and $|\Phi(m) - \Phi(\tilde m)| \leq a d(m,\tilde m)$. Recall (\ref{generator}). Then $m\mapsto \mathcal L_{K,F}^{\beta}\Phi(m)$ is bounded on $\{m\in\ell_{\lambda},\|m\|_{\lambda} \leq c\}$ for each $c>0$.
\end{lemma}
\begin{proof}
This Lemma is a straightforward consequence of the hypotheses on the kernels and Lemma \ref{ddl_distances}. Let $c>0$ be fixed, and set $A:=c^{1/\lambda}$. Notice that if $\|m\|_{\lambda}\leq c$, then for all $k\geq1$ $m_k\leq A$.\medskip

Setting $\sup_{[0,A]^2}K(x,y) = \overline K$ and $\sup_{[0,A]}F(x) = \overline F $. We use (\ref{d1c}) and (\ref{d1f}) with (\ref{inebeta}) and (\ref{Clambdabounds}), and deduce that for all $m\in\ell_{\lambda}$ such that $\|m\|_{\lambda}\leq c$,
\begin{eqnarray*}
|\mathcal L_{K,F}^{\beta}\Phi(m)|& \leq & \overline{K}\sum_{1\leq i < j < \infty}  \left| \Phi\left(c_{ij}(m)\right) - \Phi(m) \right|+ \overline{F} \sum_{i\geq 1} \int_{\Theta} \left| \Phi\left(f_{i\theta}(m)\right) - \Phi(m) \right| \beta(d\theta)\\
& \leq & a \overline{K}\sum_{1\leq i < j < \infty} d(c_{ij}(m),m)+a\overline{F}\int_{\Theta}\sum_{i\geq 1}d(f_{i\theta}(m),m)\beta(d\theta)\\
&\leq & \frac{3}{2} a \overline{K}\,\|m\|_{1} + 2a\,\overline{F}\,C^{\lambda}_{\beta} \|m\|_{1} \,\,\, \leq \,\,\, \left( \frac{3}{2} \overline{K}+ 2\overline{F} \,C^{\lambda}_{\beta}\right) a c^{1/\lambda}.
\end{eqnarray*}
\end{proof}
\noindent Finally, it remains to conclude the proof of Theorem \ref{Result_Theo}.\medskip

\noindent \textbf{Proof of Theorem \ref{Result_Theo}.} We consider the Poisson measures $N$ and $M$ as in Definition \ref{Def_Proc}., and we fix $m\in\ell_{\lambda}$. We consider $M(t) := M(m,t)$ a solution to  $SDE(K,F,m,\mathcal N,\mathcal M)$ built in Section \ref{EandUSDE}. $M$ is a strong Markov Process, since it solves a time-homogeneous Poisson-driven $S.D.E.$ We now check the points \textit{\ref{ResultTheo_i})} and \textit{\ref{ResultTheo_ii})}.\medskip

Consider any sequence $m^n\in \ell_{0+}$ such that $\lim_{n\rightarrow \infty}\delta_{\lambda}(m^n,m) = 0$ and $M^n(t) := M(m^n,t)$ the unique solution to $SDE(K,F,m^n,\mathcal N,\mathcal M_n)$ obtained in Proposition \ref{Finite_PoissonProc}. Denote by $\tau^x=\inf\{t\geq 0, \|M(m,t)\|_{\lambda}\geq x\}$ and by $\tau^x_n$ the stopping time concerning $M^n$. We will prove that for all $T\geq 0$ and $\varepsilon >0$
\begin{equation}\label{boundProbaMMn}
\lim_{n\rightarrow \infty} P\left[\sup_{[0,T]}\delta_{\lambda}\left(M(t),M^n(t)\right) > \varepsilon \right] =0.
\end{equation}

For this, consider the sequence $m^{(n)}\in\ell_{0+}$ defined by $m^{(n)}=(m_1,\cdots,m_n,0,\dots)$ and $M^{(n)}(t):= M(m^{(n)},t)$ the solution to $SDE(K,F,m^{(n)},\mathcal N,\mathcal M_n)$ obtained in Proposition \ref{Finite_PoissonProc}. and denote by $\tau^x_{(n)}$ the stopping time concerning $M^{(n)}$.\medskip

\noindent First, note that since $\lim_{n\rightarrow \infty}\delta_{\lambda}(m^{(n)},m) = \lim_{n\rightarrow \infty}\delta_{\lambda}(m^n,m) = 0$, we deduce that $\sup_{n\geq 1} \|m^{(n)}\|_{\lambda}<\infty$ and from Lemma \ref{tainfini}. that for all $t\geq 0$,
\begin{eqnarray}
\lim_{x\rightarrow \infty}\gamma_1(t,x) & = & 0 \,\,\, \textrm{where} \,\,\, \gamma_1(t,x) := \sup_{n\geq 1} P[\tau^x_{(n)}\leq t], \label{alpha1}\\
\lim_{x\rightarrow \infty}\gamma_2(t,x) & = & 0 \,\,\, \textrm{where} \,\,\, \gamma_2(t,x) := \sup_{n\geq 1} P[\tau^x_n\leq t].\label{alpha2}
\end{eqnarray}
Thus, using Proposition \ref{Result_Prop}. \textit{\ref{Prop_C01_ii})} we get for all $x>0$ 
\begin{eqnarray*}
&& P\left[\sup_{[0,T]}\delta_{\lambda}\left(M(t),M^n(t)\right) > \varepsilon \right]\\
& & \hspace{1cm}\leq \,\, P\left[\sup_{[0,T]}\delta_{\lambda}\left(M(t),M^{(n)}(t)\right) > \frac{\varepsilon}{2}\right]+ P\left[\sup_{[0,T]} \delta_{\lambda}\left(M^{(n)}(t),M^{n}(t)\right) > \frac{\varepsilon}{2} \right] \\
&& \hspace{1cm}\leq \,\,  P[\tau^x \leq T] +\gamma_1(T,x) +\frac{2}{\varepsilon} \mathbb E\left[\sup_{[0,T\wedge \tau^x_{(n)} \wedge \tau^x]}\delta_{\lambda}\left(M(t),M^{(n)}(t)\right)\right] \\
&& \hspace{1.5cm} + \gamma_1(T,x) + \gamma_2(T,x) + \frac{2}{\varepsilon} e^{C(x+1)T } \delta_{\lambda}\left(m^{(n)}, m^n \right).
\end{eqnarray*}

We first make $n$ tend to infinity and use (\ref{limE01}), then $x$ to infinity and use (\ref{alpha1}) and (\ref{alpha2}). We thus conclude that (\ref{boundProbaMMn}) holds.\medskip

We may prove point \textit{\ref{ResultTheo_ii})} using a similar computation that for \textit{\ref{ResultTheo_i})}. The proof is easier since we do not need to use a triangle inequality.\medskip

Finally, consider $(M(m,t))_{t\geq 0}$ solution to $SDE(K,F,m,\mathcal N,\mathcal M)$ and the sequence of stopping times $(\tau^{x_n})_{n\geq 1}$ where $\tau^{x_n}=\inf\{t\geq 0, \|M(m,t)\|_{\lambda}\geq x_n\}$, with $x_n = n$. 
Since $M\in\mathbb D([0,\infty),\ell_{\lambda})$, we have that  $(\tau^{x_n})_{n\geq 1}$ is non-decreasing and $\tau^{x_n} \underset{n\rightarrow \infty}{\longrightarrow} \infty$ and from Lemma \ref{LemmaLBounded}. we deduce that $(\mathcal L^{\beta}_{K,F}\Phi(M(m,s)))_{s\in[0,\tau^{x_n})}$ is uniformly bounded.\medskip

We thus apply It\^o's Formula to $\Phi(M(m,t))$ on the interval $[0,t\wedge \tau^{x_n})$ to obtain
\begin{eqnarray*}
&& \Phi(M(m,t\wedge\tau^{x_n})) - \Phi(m)\,\, = \\
 & & \int_0^{t\wedge\tau^{x_n}} \int_{i<j} \int_0^{\infty} \left[ \Phi\left(c_{ij}\left(M(m,s-)\right)\right) - \Phi\left(M(m,s-)\right)\right]  \mathds 1_{\lbrace z \leq K(M_i(m,s-),M_j(m,s-))\rbrace}\\
 && \hspace{11.3cm}\tilde{\mathcal N}(dt,d(i,j),dz) \\
&&+ \int_0^{t\wedge\tau^{x_n}} \int_{i} \int_{\Theta} \int_0^{\infty} \left[ \Phi\left(f_{i\theta}\left(M(m,s-)\right)\right) - \Phi\left(M(m,s-)\right)\right] \mathds 1_{ \lbrace z \leq F(M_i(m,s-))\rbrace}\\
 && \hspace{11.3cm} \tilde{\mathcal M}(dt,di,d\theta,dz)\\
&& + \int_0^{t\wedge\tau^{x_n}} \mathcal L_{K,F}^{\beta}\left(M(m,s)\right) ds,
\end{eqnarray*}
where $\tilde{\mathcal N}$ and $\tilde{\mathcal M}$ are two compensated Poisson measures and point \textit{\ref{ResultTheo_iii})} follows. \medskip

This ends the proof of Theorem \ref{Result_Theo}.\medskip

I would like to express my deepest thanks to my Ph.D. advisor Prof. Nicolas Fournier for his insightful comments and advices during the preparation of this work. I would like also to thank B\'en\'edicte Haas and James R. Norris for the lecture and their remarks. Finally, I sincerely thank the anonymous referee for pointing out many problems in the first version of this paper and for helping improving it.

\begin{appendix}

\section[Estimates concerning the distances]{Estimates concerning $c_{ij}$, $f_{i\theta}$, $d$ and $\delta_{\lambda}$}\label{appendix}
\setcounter{equation}{0}
Here we put all the auxiliary computations needed in Sections \ref{Proof} and \ref{EandUSDE}.\medskip

\begin{lemma}
Fix $\lambda \in (0,1]$. Consider any pair of finite permutations $\sigma$, $\tilde{\sigma}$ of $\mathbb N$. Then for all $m$ and $\tilde{m}\in \ell_{\lambda}$,
\begin{eqnarray}
d(m,\tilde{m}) &\leq &\sum_{k\geq 1}2^{-k} |m_{k} - \tilde{m}_{\tilde{\sigma}(k)}|,  \label{permutationd} \\
\delta_{\lambda}(m,\tilde{m}) &\leq &\sum_{k\geq 1} |m_{\sigma(k)}^{\lambda} - \tilde{m}^{\lambda}_{\tilde{\sigma}(k)}|.\label{lemma_permutation}
\end{eqnarray}
\end{lemma}
This lemma is a consequence of \cite[Lemma 3.1]{Sto-Coal}.\medskip

We also have the following inequality: for all $\alpha,\beta > 0$, there exists a positive constant $C = C_{\alpha,\beta}$ such that for all $x,y \geq 0$,
\begin{equation}\label{ineq_distance}
(x^{\alpha} + y^{\alpha}) |x^{\beta} - y^{\beta}| \leq 2 |x^{\alpha + \beta} - y^{\alpha + \beta} | \leq C (x^{\alpha} + y^{\alpha})  |x^{\beta}-y^{\beta}|.
\end{equation}

We now give the inequalities concerning the action of $c_{ij}$ and $f_{i\theta}$ on $\delta_{\lambda}$ and $\|\cdot\|_{\lambda}$.

\begin{lemma}
Let $\lambda \in(0,1]$ and $\theta \in \Theta$. Then for all $m$ and $\tilde{m}\in \ell_{\lambda}$, all $1\leq i < j < \infty$,

\begin{eqnarray}
\| c_{ij}(m)\|_{\lambda} & = & \|m\|_{\lambda} + (m_i + m_j)^{\lambda} - m_i^{\lambda} - m_j^{\lambda}\,\,\leq\,\,\|m\|_{\lambda}, \label{C1_01}\\
\| f_{i\theta}(m)\|_{\lambda} & = & \|m\|_{\lambda} +  m_i^{\lambda}\left(\sum_{k\geq 1} \theta_k^{\lambda} - 1 \right), \label{F1_01}\\
\delta_{\lambda}(c_{ij}(m),m) & \leq & 2m_j^{\lambda},  \label{C2_01}\\
\delta_{\lambda}(f_{i\theta}(m),m) & \leq & m_i^{\lambda}\left[\sum_{k\geq 2} \theta_k^{\lambda} + \left(1-\theta_1^{\lambda}\right)\right], \label{F2_01}\\
\delta_{\lambda}(c_{ij}(m),c_{ij}({\tilde{m}})) & \leq & \delta_{\lambda}(m,\tilde{m}), \label{C3_01} \\
\delta_{\lambda}(f_{i\theta}(m),f_{i\theta}({\tilde{m}})) & \leq & \delta_{\lambda}(m,\tilde{m}) + |m_i^{\lambda} - \tilde{m}_i^{\lambda}|\left(\sum_{k\geq 1} \theta_k^{\lambda} - 1 \right).  \label{F3_01}
\end{eqnarray}
On the other hand, recall (\ref{projector}), we have, for $u,v\in \mathbb N$ with $1\leq u < v$,
\begin{eqnarray}
\label{dln} \\
\delta_{\lambda}(f_{i\psi_u(\theta)}(m),f_{i\psi_{v}(\theta)}(m)) & \leq & \sum_{k=u+1}^v \theta_{k}^{\lambda} m_i^{\lambda}.\nonumber
\end{eqnarray}
\end{lemma}
Note that in the case $\sum_{k\geq 1} \theta_k^{\lambda} - 1 <0$, we have that $\|\cdot\|_{\lambda}$ and $\delta_{\lambda}$ are respectively, decreasing and contracting under the action of fragmentation and the calculations in precedent sections would be simpler.

\begin{proof}

\noindent First (\ref{C1_01}) and (\ref{F1_01}) are evident. Next, (\ref{C2_01}) and (\ref{C3_01}) are proved in \cite[Lemma A.2]{Sto-Coal2}. \medskip

To prove (\ref{F2_01}) let $\theta = (\theta_1,\cdots) \in \Theta$, $i\geq1$ and $p\geq 2$ and set $l := l(m) = \min\{k\geq 1:m_k \leq \theta_p\,m_i\}$, we consider the largest particle of the original system (before dislocation of $m_i$) that is smaller than the $p$-th fragment of $m_i$, this is $m_l$. Consider now $\sigma$, the finite permutation of $\mathbb N$ that achieves:

\begin{eqnarray}\label{sigma}
&&\\
&&(f_k)_{k\geq 1} \,\, := \,\, \left(\left[f_{i\theta}(m)\right]_{\sigma(k)}\right)_{k\geq 1} \nonumber\\
&&=\,\, (m_1,\cdots,m_{i-1},\theta_1 m_i,m_{i+1},\cdots, m_{l-1},m_{l},\theta_2 m_i,\theta_3 m_{i},\cdots, \theta_p m_i, \left[f_{i\theta}(m)\right]_{l+1}, \cdots ).  \nonumber
 \end{eqnarray}
It suffices to compute the $\delta_{\lambda}$-distance of the sequences $(f_k)_k$ and $(m_k)_k$:
\begin{equation}
 \begin{array}{cccccccccccccc}
 m_1&\cdots&m_{i-1}&\theta_1 m_i&m_{i+1}&\cdots& m_{l-1}& m_{l}&\theta_2 m_i& \theta_3 m_{i}&\cdots& \theta_p m_i &f_{l+p} & \cdots  \\
 m_1&\cdots&m_{i-1}& m_i& m_{i+1}&\cdots& m_{l-1}& m_l& m_{l+1}& m_{l+2}&\cdots &m_{l+p-1} &m_{l+p}& \cdots  
 \end{array}
\end{equation}
Thus, using (\ref{lemma_permutation}), we have
\begin{eqnarray*}
\delta_{\lambda}(f_{i\theta}(m),m) & \leq & \sum_{k\geq 1} \left|  f_k^{\lambda} - m_k^{\lambda} \right| \,\,\, = \,\,\, \left(\sum_{k = 1}^{l}+ \sum_{k = l+1}^{l+p-1}+\sum_{k \geq l+p} \right)\left| f_k^{\lambda}  - m_k^{\lambda} \right| \\
& \leq &  (1-\theta_1^{\lambda})m_i^{\lambda} + \sum_{k = l+1}^{l+p-1} \left| \theta_{k-l+1}^{\lambda} m_i^{\lambda} - m_k^{\lambda} \right|  + \sum_{k \geq l+p} \left|  f_{k}^{\lambda} - m_k^{\lambda}  \right|\\
& \leq &  (1-\theta_1^{\lambda})m_i^{\lambda} + \left(\sum_{k=2}^p \theta_k^{\lambda} m_i^{\lambda} +\sum_{k = l+1}^{l+p-1}  m^{\lambda}_{k}\right) +\sum_{k\geq l+p} \left(  f^{\lambda}_k + m^{\lambda}_{k}\right)\\
& = & (1-\theta_1^{\lambda})m_i^{\lambda}  + m_i^{\lambda} \sum_{k=2}^{\infty} \theta_k^{\lambda}   + 2\sum_{k > l} m^{\lambda}_{k}.
\end{eqnarray*}
For the last equality it suffices to remark that $\sum_{k\geq l} f^{\lambda}_k$ contains all the remaining fragments of $m_i^{\lambda}$ and all the particles $m_k^{\lambda}$ with $k>l$.\medskip

\noindent Note that if $m\in \ell_{0+}$ the last sum consists of a finite number of terms and it suffices to take $p$ large enough (implying $l$ large) to cancel this term. On the other hand, if $m \in \ell_{\lambda} \setminus \ell_{0+}$ then the last sum is the tail of a convergent serie and since $l \rightarrow \infty$ whenever $p\rightarrow \infty$, we conclude by making $p$ tend to infinity and (\ref{F2_01}) follows. \medskip

\noindent To prove (\ref{F3_01}) consider $\tilde m$, $l := l(m) \vee l(\tilde m)$ and the permutations $\sigma$ and $\tilde \sigma$ associated to this $l$, exactly as in (\ref{sigma}). Let $f$ and $\tilde f$  be the corresponding objects concerning $m$ and $\tilde m$:
\begin{equation}\label{sigma_aux}
 \begin{array}{cccccccccccccc}
 m_1&\cdots&m_{i-1}&\theta_1 m_i&m_{i+1}&\cdots& m_{l-1}&m_{l}&\theta_2 m_i& \theta_3 m_{i}&\cdots& \theta_p m_i &f_{l+p} & \cdots  \\
\tilde m_1&\cdots&\tilde m_{i-1}&\theta_1 \tilde m_i&\tilde m_{i+1}&\cdots& \tilde m_{l-1}&\tilde m_{l}&\theta_2 \tilde m_i& \theta_3 \tilde m_{i}&\cdots& \theta_p \tilde m_i &\tilde f_{l+p} & \cdots  \\ 
 \end{array}
\end{equation}

Using again (\ref{lemma_permutation}) for $(f_k)_{k}$ and $ (\tilde f_k)_{k}$, we have
\begin{eqnarray*}
&&\delta_{\lambda}(f_{i\theta}(m),f_{i\theta}(\tilde m))\\
&&\hspace{5mm} \leq  \sum_{k\geq 1} \left|  f_k^{\lambda} - \tilde f_k^{\lambda} \right| \,\,\,= \,\,\,\left(\sum_{k = 1}^{l}+ \sum_{k = l+1}^{l+p-1}+\sum_{k \geq l+p} \right)\left| f_k^{\lambda}  - \tilde f_k^{\lambda} \right| \\
&&\hspace{5mm} =  \sum_{k = 1}^{l} \left| m_k^{\lambda}  - \tilde m_k^{\lambda} \right| - \left| m_i^{\lambda}  - \tilde m_i^{\lambda} \right|  + \sum_{k = 1}^{p} \theta_k^{\lambda}\left| m_i^{\lambda}  - \tilde m_i^{\lambda} \right| + \sum_{k \geq l+p} \left( f_k^{\lambda}  + \tilde f_k^{\lambda} \right)
\end{eqnarray*}
\begin{eqnarray*}
&&\hspace{5mm} =   \sum_{k = 1}^{l} \left| m_k^{\lambda}  - \tilde m_k^{\lambda} \right| - \left| m_i^{\lambda}  - \tilde m_i^{\lambda} \right|  + \sum_{k = 1}^{p} \theta_k^{\lambda}\left| m_i^{\lambda}  - \tilde m_i^{\lambda} \right|+ \sum_{k > p} \theta_k^{\lambda} \left( m^{\lambda}_i + \tilde m^{\lambda}_i \right) + \sum_{k > l} \left( m^{\lambda}_{k}  + \tilde  m^{\lambda}_{k}\right) \\
&&\hspace{5mm} = \sum_{k = 1}^{l} \left| m_k^{\lambda}  - \tilde m_k^{\lambda} \right|  + |m_i^{\lambda} - \tilde{m}_i^{\lambda}|\left(\sum_{k= 1}^p \theta_k^{\lambda} - 1 \right) + \left( m^{\lambda}_i + \tilde m^{\lambda}_i \right) \sum_{k > p} \theta_k^{\lambda}+ \sum_{k > l} \left( m^{\lambda}_{k}  + \tilde  m^{\lambda}_{k}\right).
\end{eqnarray*}
Notice that the last two sums are the tails of convergent series, note also that $l\rightarrow \infty$ whenever $p\rightarrow \infty$. We thus conclude making $p$ tend to infinity.\medskip

Finally, to prove (\ref{dln}) we consider the permutation $\sigma$ as in (\ref{sigma}) with $p=v$  and  $l:= l(m)$. Recall (\ref{sigma_aux}), we have
\begin{eqnarray*}
\delta_{\lambda}(f_{i\psi_u(\theta)}(m),f_{i\psi_{v}(\theta)}(m)) & = & \delta_{\lambda}(f_{i\psi_{v}(\psi_u(\theta))}(m),f_{i\psi_{v}(\theta)}(m)) \\
&\leq &\sum_{k \geq 1} \left| [f_{i\psi_{v}(\psi_u(\theta))}(m)]_{\sigma(k)}^{\lambda} -[f_{i\psi_{v}(\theta)}(m)]_{\sigma(k)}^{\lambda}\right|\\
& \leq & \sum_{k = u+1}^{v} \theta_k^{\lambda} m_i^{\lambda} + 2\sum_{k > l} m_k^{\lambda}. 
\end{eqnarray*}
We used that $[\psi_{v}(\psi_{u}(\theta))]_{k} = 0$ for $k=u+1,\cdots,v$. Since $m\in \ell_{\lambda}$, we conclude making $l$ tend to infinity. 

\end{proof}

\begin{lemma}\label{ddl_distances}
Consider $m$, $\tilde m \in S^{\downarrow}$ and $1\leq i < j < \infty$. Recall the definition of $d$ (\ref{distance1}), $\delta_{\lambda}$  (\ref{def_distance}), $c_{ij}(m)$ and $f_{i\theta}(m)$  (\ref{c_and_f}) and $\psi_n(\theta)$  (\ref{projector}). For $\lambda \in (0,1)$ and for all $m, \tilde m \in \ell_{\lambda}$ there exists a positive constant $C$ depending on $\lambda$ such that
\begin{equation}\label{d_dlambda}
d(m,\tilde{m})\leq \delta_1(m,\tilde{m}) \leq C (\|m\|_1^{1-\lambda} \vee \|\tilde m\|_1^{1-\lambda})\,\delta_{\lambda}(m,\tilde{m}).
\end{equation}
\noindent Next, 
\begin{eqnarray}
d(c_{ij}(m),m) \leq \frac{3}{2}2^{-i}m_j, && \sum_{1\leq k < l <\infty}d(c_{kl}(m),m) \leq \frac{3}{2} \|m\|_1, \label{d1c}\\ 
d(c_{ij}(m),c_{ij}(\tilde m))& \leq &(2^{i} + 2^{j})d(m,\tilde m). \label{d2c}\\
d(f_{i\theta}(m),m) &\leq& 2(1-\theta_1)2^{-i}m_i,  \label{d1f}\\ 
d(f_{i\theta}(m),f_{i\theta}(\tilde m)) &\leq & C (\|m\|_1^{1-\lambda} \vee \|\tilde m\|_1^{1-\lambda})\,\delta_{\lambda}(m,\tilde{m}),  \label{d2f}\\
d(f_{i\theta}(m),f_{i\psi_n(\theta)}(m)) &\leq & m_i \sum_{k>n}\theta_k. \label{d3f} 
\end{eqnarray}

\end{lemma}
\begin{proof} 
The first inequality in (\ref{d_dlambda}) follows readily from the definition of $d$ and the second one comes from (\ref{ineq_distance}), with $\alpha = 1-\lambda$ and $\beta = \lambda$. The inequalities (\ref{d1c}) and (\ref{d2c}) involving $d$ are proved in \cite[Corollary 3.2.]{Sto-Coal}.\medskip

We prove (\ref{d1f}) exactly as (\ref{F2_01}). Consider $p$, $l$ and the permutation $\sigma$ defined by (\ref{sigma}), from (\ref{permutationd}) and since $i\leq l+1 \leq l +p$, we obtain 
\begin{eqnarray*}
d(f_{i\theta}(m),m) & \leq &  \left(\sum_{k = 1}^{l}+ \sum_{k = l+1}^{l+p-1}+\sum_{k \geq l+p} \right)2^{-k}\left| f_k- m_k \right| \\
& \leq &  (1-\theta_1)2^{-i}m_i + \sum_{k = l+1}^{l+p-1}2^{-k} \left| \theta_{k-l+1} m_i- m_k \right|  + \sum_{k \geq l+p}2^{-k} \left|  f_{k} - m_k \right| \\
& \leq & (1-\theta_1)2^{-i}m_i  + \left(\sum_{k=2}^p2^{-i} \theta_k m_i +\sum_{k = l+1}^{l+p-1}  m_{k}\right) +\sum_{k\geq l+p} 2^{-i} \left(  f_k + m_{k}\right)\\
& \leq & (1-\theta_1)2^{-i}m_i   + 2^{-i} m_i \sum_{k=2}^{\infty} \theta_k  + 2\sum_{k > l} m_{k}.
\end{eqnarray*}
Since $m\in\ell_1$, we conclude using (\ref{Hyp1_Beta}) and making $l$ tend to infinity.\medskip

Next, we prove (\ref{d2f}) as (\ref{F3_01}) using $\delta_1$. Consider $p$, $l$ and the permutations $\sigma$ and $\tilde \sigma$ defined by (\ref{sigma}). Recall (\ref{sigma_aux}), using (\ref{d_dlambda}) then (\ref{lemma_permutation}) (applied to $\delta_1$) and since, $i\leq l+1 \leq l +p$ we obtain 
\begin{eqnarray*}
&&d(f_{i\theta}(m),f_{i\theta}(\tilde m))\\
&& \hspace{1cm}\leq \delta_1(f_{i\theta}(m),f_{i\theta}(\tilde m))\,\,
 \leq \,\, \left(\sum_{k = 1}^{l}+ \sum_{k = l+1}^{l+p-1}+\sum_{k \geq l+p} \right)\left| f_k  - \tilde f_k \right| \\
&&\hspace{1cm} \leq \sum_{k = 1}^{l}\left| m_k  - \tilde m_k \right| + (\theta_1-1)\left| m_i - \tilde m_i \right|+ \sum_{k = l+1}^{l+p-1} \theta_{k-l+1}\left| m_i  - \tilde m_i \right| + \sum_{k \geq l+p} \left( f_k  + \tilde f_k \right)  \\
&&\hspace{1cm} \leq   \sum_{k = 1}^{l} \left| m_k - \tilde m_k \right|  + \left| m_i - \tilde m_i \right|  \left( \sum_{k = 1}^{p} \theta_k -1\right) +  \left( m_i + \tilde m_i \right)\sum_{k > p} \theta_k + \sum_{k > l} \left( m_{k}  + \tilde  m_{k}\right)\\
& &\hspace{1cm} \\
&&\hspace{1cm} \leq   \sum_{k = 1}^{l} \left| m_k - \tilde m_k \right| +  \left( m_i + \tilde m_i \right)\sum_{k > p} \theta_k + \sum_{k > l} \left( m_{k}  + \tilde  m_{k}\right).
\end{eqnarray*}
We used that for $k\geq l+p$, $f_k$ contains all the remaining fragments of $m_i$ and the particles $m_j$ with $j>l$ and (\ref{Hyp1_Beta}). Since $m,\tilde m \in\ell_1$ we conclude making $p$ tend to infinity and using (\ref{d_dlambda}).\medskip

Finally, for inequality (\ref{d3f}), let $i\geq 1$, $p \geq 1$ and  $l:=l_p(m)= \min\{k\geq 1:m_k \leq (\theta_n/p)m_i\}$ and consider $\sigma$, the finite permutation of $\mathbb N$ that achieves:
\begin{eqnarray}
(f_k)_{k\geq 1} & := & \left(\left[f_{i\theta}(m)\right]_{\sigma(k)}\right)_{k\geq 1}  \nonumber\\
 & = & (m_1,\cdots,m_{i-1},\theta_1 m_i, \cdots, \theta_n m_i,m_{i+1},\cdots, m_{l-1},m_{l}, \left[f_{i\theta}(m)\right]_{l+n}, \cdots ). \label{sigma2}
 \end{eqnarray}
Thus, from (\ref{d_dlambda}) and (\ref{lemma_permutation}), and since $i\leq l+1 \leq l+n+1$, we deduce 
\begin{eqnarray*}
d\left((f_{i\theta}(m),f_{i\psi_n(\theta)}(m)\right) & \leq &  \delta_1\left((f_{i\theta}(m),f_{i\psi_n(\theta)}(m)\right) \,\,
 = \,\,  \sum_{k \geq 1}\left| [f_{i\theta}(m)]_{k} - [f_{i\psi_n(\theta)}(m)]_{k}  \right| \\
& \leq &  \left(\sum_{k = 1}^{l}+ \sum_{k =l+1}^{l+n-1}+\sum_{k \geq l+n} \right) \left| [f_{i\theta}(m)]_{\sigma(k)} - [f_{i\psi_n(\theta)}(m)]_{\sigma(k)}  \right| \\
& \leq &  \sum_{k > n} \theta_k m_i + 2  \sum_{k> l} m_k.
\end{eqnarray*}
The last sum being the tail of a convergent series we conclude making $l\rightarrow \infty$.\medskip

This concludes the proof of Lemma \ref{ddl_distances}.
\end{proof}
\end{appendix}

\end{document}